\documentclass[12pt]{article}
\usepackage{amsmath,amsxtra,latexsym,amsthm,amssymb,amscd,amsfonts}
\usepackage[portrait, top=3.5cm, bottom=3cm, left=3.5cm, right=2cm] {geometry}
\usepackage{graphicx}
\usepackage{color}
\usepackage{enumerate}

\theoremstyle{plain}
\newtheorem{proposition}{\bf Proposition}[section]

\newtheorem{theorem}{\bf Theorem}[section]

\def\tto{\;{\lower 1pt \hbox{$\rightarrow$}}\kern -10pt
\hbox{\raise 2pt \hbox{$\rightarrow$}}\;}

\def\R{{\rm I\!R}}
\def\N{{\rm I\!N}}

\def\ov{\bar{v}}

\def\co{\mbox{\rm co}\,}

\def\gph{\mbox{\rm gph}\,}
\def\epi{\mbox{\rm epi}\,}

\begin{document}
\pagestyle{myheadings}

\newtheorem{Theorem}{Theorem}[section]
\newtheorem{Proposition}[Theorem]{Proposition}
\newtheorem{Remark}[Theorem]{Remark}
\newtheorem{Lemma}[Theorem]{Lemma}
\newtheorem{Corollary}[Theorem]{Corollary}
\newtheorem{Definition}[Theorem]{Definition}
\newtheorem{Example}[Theorem]{Example}
\renewcommand{\theequation}{\thesection.\arabic{equation}}
\normalsize

\setcounter{equation}{0}

\title{Analyzing a Maximum Principle for Finite Horizon State Constrained Problems via Parametric Examples. Part 1: Problems with Unilateral State Constraints\footnote{Financial supports from several research projects in Taiwan and Vietnam are gratefully acknowledged.}}

\author{V.T.~Huong\footnote{Institute of Mathematics, Vietnam Academy of Science and Technology, 18 Hoang
Quoc Viet, Hanoi 10307, Vietnam; email: vthuong@math.ac.vn; huong263@gmail.com.},\ \, J.-C.~Yao\footnote{Center for General Education, China Medical University, 
	Taichung 40402, Taiwan; Email: yaojc@mail.cmu.edu.tw},\ \, and\ \, N.D. Yen\footnote{Institute
of Mathematics, Vietnam Academy of Science and Technology, 18 Hoang
Quoc Viet, Hanoi 10307, Vietnam; email: ndyen@math.ac.vn.}}

\maketitle
\date{}%\small\today}
\centerline{\textit{(Dedicated to Professor Do Sang Kim on the occasion of his 65th birthday)}}

\bigskip
\begin{quote}
\noindent {\bf Abstract.} In the present paper, the maximum principle for finite horizon state constrained problems from the book by R. Vinter [\textit{Optimal Control}, Birkh\"auser, Boston, 2000; Theorem~9.3.1] is analyzed via parametric examples. The latter has origin in a recent paper by  V.~Basco, P.~Cannarsa, and H.~Frankowska, and resembles the optimal growth problem in mathematical economics. The solution existence of these parametric examples is established by invoking Filippov's existence theorem for Mayer problems. Since the maximum principle is only a necessary condition for local optimal processes, a large amount of additional investigations is needed to obtain a comprehensive synthesis of finitely many processes suspected for being local minimizers. Our analysis not only helps to understand the principle in depth, but also serves as a sample of applying it to meaningful prototypes of  economic optimal growth models. Problems with unilateral state constraints are studied in Part 1 of the paper. Problems with bilateral state constraints will be addressed in Part 2. 



\end{quote}
\section{Introduction}
It is well known that optimal control problems with state constraints are models of importance, but one usually faces with a lot of difficulties in analyzing them. These models have been considered since the early days of the optimal control theory. For instance, the whole Chapter VI of the classical work \cite[pp.~257--316]{Pont_Bolt_Gamk_Mish_1962} is devoted to problems with restricted phase coordinates. There are various forms of the maximum principle for optimal control problems with state constraints; see, e.g., \cite{Hartl_Sethi_Vickson_1995}, where the relations between several forms are shown and a series of numerical illustrative examples have been solved.

To deal with state constraints, one has to use functions of bounded variation, Borel measurable functions, Lebesgue-Stieltjes integral, nonnegative measures on the $\sigma-$algebra of the Borel sets, the Riesz Representation Theorem for the space of continuous functions, and so on.

By using the maximum principle presented in \cite[pp.~233--254]{Ioffe_Tihomirov_1979}, Phu \cite{Phu_1989,Phu_1992} has proposed  an ingenious method called \textit{the method of region analysis} to solve several classes of optimal control problems with one state and one control variable, which have both state and control constraints. Minimization problems of the Lagrange type were considered by the author and, among other things, it was assumed that integrand of the objective function is strictly convex with respect to the control variable. To be more precise, the author considered \textit{regular problems}, i.e., the  optimal control problems where the Pontryagin function is strictly convex with respect to the control variable. 

In the present paper, the maximum principle for finite horizon state constrained problems from the book by Vinter \cite[Theorem~9.3.1]{Vinter_2000} is analyzed via parametric examples. The latter has origin in a recent paper by Basco, Cannarsa, and Frankowska \cite[Example~1]{Basco_Cannarsa_Frankowska_2018}, and resembles the optimal growth problem in mathematical economics (see, e.g., \cite[pp.~617--625]{Takayama_1974}). The solution existence of these parametric examples, which are \textit{irregular optimal control problems} in the sense of Phu \cite{Phu_1989,Phu_1992}, is established by invoking Filippov's existence theorem for Mayer problems \cite[Theorem~9.2.i and Section~9.4]{Cesari_1983}. Since the maximum principle is only a necessary condition for local optimal processes, a large amount of additional investigations is needed to obtain a comprehensive synthesis of finitely many processes suspected for being local minimizers. Our analysis not only helps to understand the principle in depth, but also serves as a sample of applying it to meaningful prototypes of  economic optimal growth models. 

Note that the \textit{maximum principle} for finite horizon state constrained problems in \cite[Chapter~9]{Vinter_2000} covers many known ones for smooth problems and allows us to deal with nonsmooth problems by using the \textit{Mordukhovich normal cone} and the \textit{Mordukhovich subdifferential} \cite{Mordukhovich_2006a,Mordukhovich_2006b,Mordukhovich_2018}, which are also called the limiting normal cone and the limiting subdifferential. This principle is a necessary optimality condition which asserts the existence of a multipliers set $(p, \mu, \nu,\gamma)$ consisting of an \textit{absolutely continuous function}  $p$, a \textit{function of bounded variation}~$\mu$, a \textit{Borel measurable function} $\nu$, and a \textit{real number} $\gamma\geq 0$, where $(p,\mu,\gamma)\neq (0,0,0)$, such that the four conditions (i)--(iv) in Theorem~\ref{V_thm9.3.1 necessary condition} below are satisfied. The relationships between these conditions are worthy a detailed analysis. We will present such an analysis via three parametric examples of optimal control problems of the Langrange type, which have five parameters $(\lambda,a,x_0,t_0,T)$, where $\lambda>0$ appears in the description of the objective function, $a>0$ appears in the differential equation, $x_0$ is the initial value, $t_0$ is the initial time, and $T$ is the terminal time. Observe that,  in Example~1 of \cite{Basco_Cannarsa_Frankowska_2018}, $T=\infty$, $x_0$ and $t_0$ are fixed. Problems with unilateral state constraints are studied in Part 1 of the paper. Problems with bilateral state constraints will be addressed in Part 2. 

This Part 1 is organized as follows. Section \ref{Background Materials} presents some background materials including the above-mentioned maximum principle and Filippov's existence theorem for Mayer problems. Control problems without state constraints are considered in Section \ref{Example 1}, while control problems with unilateral state constraints are studied in Section \ref{Example 2}. Some concluding remarks are given in Section \ref{Conclusions}.

\section{Background Materials}\label{Background Materials}
In this section, we give some notations, definitions, and results that will be used repeatedly in the sequel. 

\subsection{Notations and Definitions }
 The symbol $\R$ (resp., $\N)$ denotes the set of real numbers (resp., the set of positive integers). The norm in the $n$-dimensional Euclidean space $\R^n$ is denoted by $\|.\|$. For a subset $C\subset\R^n$, we abbreviate its \textit{convex hull} to $\co C$. For a  set-valued map $F: \R^n \rightrightarrows \R^m$,  we call the set ${\rm gph}\, F := \{(x,y) \in \R^n \times \R^m\,:\, y\in F(x) \}$ the \textit{graph} of $F$.
 
 Let $\Omega \subset \R^n$ be a closed set and $\ov \in \Omega$. The {\it Fr\'echet normal cone} (also called the {\it prenormal
 	cone}, or the {\it regular normal cone}) to
 $\Omega\subset\R^n$ at $\ov$ is given by
 \begin{eqnarray*}
  \widehat
 	N_{\Omega}(\ov)=\left\{v'\in\R^n\,:\,
 	\displaystyle\limsup_{v\xrightarrow{\Omega}\ov}\,\displaystyle\frac{\langle
 		v',v-\ov \rangle}{\|v-\ov \|}\leq 0\right\},\end{eqnarray*} where
 $v\xrightarrow{\Omega}\ov$ means $v\to \ov$ with $v\in\Omega$.  The {\it Mordukhovich} (or {\it limiting}) {\it normal cone} to $\Omega$ at $\ov$ is defined by
 \begin{eqnarray*} N_{\Omega}(\ov)
 		=\big\{v'\in\R^n\,:\,\exists \mbox{ sequences } v_k\to \ov,\ v_k'\rightarrow v' \mbox{ with } v_k'\in \widehat
 		N_{\Omega}(v_k)\; \mbox{for all}\; k\in \N\big\}.\end{eqnarray*}
Given an extended real-valued function $\varphi: \R^n \rightarrow \R\cup\{-\infty, +\infty\}$, one defines the \textit{epigraph} of $\varphi$ by  $\epi \varphi=\{(x, \mu)\in \R^n\times \R \,:\, \mu \geq \varphi (x)\}$. The \textit{Mordukhovich subdifferential} (or \textit{limiting subdifferential}) of $\varphi$ at $\bar x\in \R^n$ with $|\varphi (\bar x)|< \infty$ is defined by
\begin{equation*}
\partial\varphi(\bar x)=\big\{x^*\in \R^n \;:\; (x^*, -1)\in N\big((\bar x, \varphi(\bar x)); \epi \varphi \big)\big\}.
\end{equation*}
If $|\varphi (x)|= \infty$, then one puts $\partial\varphi(\bar x)=\emptyset$.
The reader is referred to \cite[Chapter~1]{Mordukhovich_2006a} and \cite[Chapter~1]{Mordukhovich_2018} for comprehensive treatments of the Fr\'echet normal cone, the limiting normal cone, the limiting subdifferential, and the related calculus rules.
 
For a given segment $[t_0, T]$ of the real line, we denote the $\sigma$-algebra of its Lebesgue measurable subsets (resp., the $\sigma$-algebra of its Borel measurable subsets) by $\mathcal{L}$ (resp., $\mathcal{B}$). The Sobolev space  $W^{1,1}([t_0, T], \R^n)$ is the linear space of the absolutely continuous functions $x:[t_0, T] \to \R^n$ endowed with the norm $\|x\|_{W^{1,1}}=\|x(t_0)\|+\displaystyle\int_{t_0}^T \|\dot x(t)\| dt$ (see, e.g., \cite[p.~21]{Kolmogorov_Fomin_1970} for this and another equivalent norm).

As in \cite[p.~321]{Vinter_2000}, we consider the following \textit{finite horizon optimal control problem of the Mayer type}, denoted by $\mathcal M$,
\begin{equation}\label{cost functional_FP}
\mbox{Minimize}\ \; g(x(t_0), x(T)),
\end{equation}
over $x \in W^{1,1}([t_0, T], \R^n)$  and measurable functions $u:[t_0, T] \to \R^m$ satisfying
\begin{equation}\label{state control system_FP}
	\begin{cases}
		\dot x(t)=f(t, x(t), u(t)),\quad &\mbox{a.e.\ } t\in [t_0, T]\\
		(x(t_0), x(T))\in C\\
		u(t)\in U(t), &\mbox{a.e.\ } t\in [t_0, T]\\
		h(t, x(t))\leq 0, & \forall t\in [t_0, T],
	\end{cases}
\end{equation}
where $[t_0, T]$ is a given interval, $g: \R^n\times \R^n \to \R$, $f: [t_0, T]\times \R^n\times \R^m \to \R^n$, and $h:[t_0, T]\times \R^n \to \R$ are given functions, $C\subset \R^n\times \R^n$ is a closed set, and $U: [t_0, T]\rightrightarrows \R^m$ is a set-valued map. 

A measurable function $u:[t_0, T] \to \R^m$ satisfying $u(t)\in U(t)$ a.e. $t\in [t_0, T]$ is called a \textit{control function}. A \textit{process} $(x, u)$ consists of a control function $u$ and an arc $x \in W^{1,1}([t_0, T]; \R^n)$ that is a solution to the differential equation in \eqref{state control system_FP}. A \textit{state trajectory} $x$ is the first component of some process $(x, u)$. A process $(x, u)$ is called \textit{feasible} if the state trajectory satisfies the \textit{endpoint constraint} $(x(t_0), x(T))\in C$ and the \textit{state constraint} $h(t, x(t))\leq 0$ for all $t\in [t_0, T]$.

Due to the appearance of the state constraint, the problem $\mathcal M$ in \eqref{cost functional_FP}--\eqref{state control system_FP} is said to be an \textit{optimal control problem with state constraints}. But, if the inequality $h(t, x(t))\leq 0$ is fulfilled for every $(t, x(t))$ with $t \in [t_0, T]$ and $x \in W^{1,1}([t_0, T]; \R^n)$ (for example, when $h$ is constant function having a fixed nonpositive value), i.e., the condition $h(t, x(t))\leq 0$ for all  $t\in [t_0, T]$ can be removed from~\eqref{state control system_FP}, then one says that $\mathcal M$ an \textit{optimal control problem without state constraints}.

 The \textit{Hamiltonian} $\mathcal H: [t_0, T]\times \R^n \times \R^n\times \R^m \to \R$ of \eqref{state control system_FP} is defined by 
\begin{equation}\label{Hamiltonian}
\mathcal H(t, x, p, u):=p.f(t, x, u)=\displaystyle\sum_{i=1}^np_if_i(t, x, u). 
\end{equation}

\begin{Definition}\label{local_minimizer} {\rm 	
	A feasible process $(\bar x, \bar u)$ is called a $W^{1,1}$ \textit{local minimizer} for  $\mathcal M$ if there exists $\delta>0$ such that $g(\bar x(t_0), \bar x(T))\leq g(x(t_0), x(T))$ for any feasible processes $(x, u)$ satisfying $\|\bar x-x\|_{W^{1,1}} \leq \delta$.}
	\end{Definition}
	
	\begin{Definition}\label{global_minimizer} {\rm 	
			A feasible process $(\bar x, \bar u)$ is called a $W^{1,1}$ \textit{global minimizer} for  $\mathcal M$ if, for any feasible processes $(x, u)$, one has $g(\bar x(t_0), \bar x(T))\leq g(x(t_0), x(T))$.}
	\end{Definition}
	
\begin{Definition}[{See \cite[p.~329]{Vinter_2000}}]\label{partial_hybrid subdiff} {\rm 
The \textit{partial hybrid subdifferential} $\partial^>_x h(t, x)$ of $h(t,x)$ w.r.t. $x$ is given by
\begin{align}\label{h-partial subdiff}
\partial^>_x h(t, x):=\co\big\{\xi \,:\, &\mbox{ there exists }(t_i, x_i)\overset{h}{\rightarrow}(t, x) \mbox{ such that }\nonumber\\
&\ h(t_k, x_k)>0 \mbox{ for all } k \mbox{ and }\nabla_xh(t_k, x_k)\to \xi\big\}, 
\end{align} where the symbol $(t_k, x_k)\overset{h}{\rightarrow}(t, x)$ means that $(t_k, x_k)\rightarrow (t, x)$ and $h(t_k, x_k)\rightarrow h(t, x)$ as $k\to\infty$.}
\end{Definition}

\subsection{A Maximum Principle for State Constrained Problems}

Due to the appearance of the state constraint $h(t, x(t))\leq 0$ in $\mathcal{M}$, one has to introduce a multiplier that is an element in the topological dual $C^*([t_0, T]; \R)$ of the space of continuous functions $C([t_0, T]; \R)$ with the supremum norm. By the Riesz Representation Theorem (see, e.g., \cite[Theorem~6, p.~374]{Kolmogorov_Fomin_1970} and \cite[Theorem~1, pp.~113--115]{Luenberger_1969}), any bounded linear 
functional $f$ on $C([t_0, T]; \R)$ can be uniquely represented in the form 
\begin{equation*}
f(x) =\int_{[t_0,T]} x(t) dv(t),
\end{equation*}
where $v$ is \textit{a function of bounded variation} on $[t_0, T]$ which vanishes at $t_0$ and which are continuous from the right at every point $\tau\in (t_0, T)$, and $\displaystyle\int_{[t_0,T]} x(t) dv(t)$ is the Riemann-Stieltjes integral of $x$ with respect to $v$ (see, e.g., \cite[p.~364]{Kolmogorov_Fomin_1970}). 
The set of the elements of $C^*([t_0, T]; \R)$ which are given by nondecreasing functions $v$ is denoted by $C^\oplus(t_0, T)$.

Every $v \in C^*([t_0, T]; \R)$ corresponds to \textit{a finite regular measure}, denoted by $\mu_v$, on the $\sigma$-algebra ${\mathcal B}$ of the Borel subsets of $[t_0, T]$ by the formula
	\begin{equation*}
	\mu_v(A) :=\int_{[t_0,T]} \chi_A(t) dv(t),
	\end{equation*} where $\chi_A(t)=1$ for $t\in A$ and $\chi_A(t)=0$ if $t\notin A$. Due to the correspondence $v\mapsto\mu_v$, we call every element $v\in C^*([t_0, T]; \R)$ a ``measure" and identify $v$ with $\mu_v$. Clearly,  the measure corresponding to each $v\in C^\oplus(t_0, T)$ is nonnegative.
	
	The integrals $\displaystyle\int_{[t_0, t)}\nu(s)d\mu(s)$ and $\displaystyle\int_{[t_0, T]}\nu(s)d\mu(s)$ of a Borel measurable function $\nu$ in next theorem are understood in the sense of the Lebesgue-Stieltjes integration \cite[p.~364]{Kolmogorov_Fomin_1970}. 

\begin{theorem}[{See \cite[Theorem~9.3.1]{Vinter_2000}}]\label{V_thm9.3.1 necessary condition}
Let $(\bar x, \bar u)$ be a $W^{1,1}$ local minimizer for $\mathcal M$. Assume that for some $\delta>0$, the following hypotheses are satisfied:
\begin{enumerate}[\rm (H1)]
\item $f(., x, .)$ is $\mathcal{L}\times \mathcal{B}^m$ measurable, for fixed $x$. There exists a Borel measurable function $k(., .):[t_0, T]\times \R^m \to \R$ such that $t \mapsto k(t, \bar u(t))$ is integrable and 
\begin{equation*}
\| f(t, x, u)-f(t, x', u)\|\leq k(t, u)\|x-x'\|, \quad \forall x, x'\in \bar x(t)+\delta\bar B,\; \forall u \in U(t)
\end{equation*} for almost all $t\in [t_0, T]$;
\item $\gph U$ is a Borel set in $[t_0,T]\times\R^m$;
\item $g$ is Lipschitz continuous on the ball $(\bar x(t_0), \bar x(T))+\delta\bar B$;
\item $h$ is upper semicontinuous and there exists $K>0$ such that 
\begin{equation*}
\| h(t, x)-h(t, x')\|\leq K\|x-x'\|, \quad \forall x, x'\in \bar x(t)+\delta\bar B,\; \forall t \in [t_0, T].
\end{equation*}
\end{enumerate}
Then there exist $p\in W^{1,1}([t_0, T]; \R^n)$, $\gamma \geq 0$, $\mu \in C^\oplus(t_0, T)$, and a Borel measurable  function $\nu:[t_0, T]\to  \R^n$ such that $(p, \mu, \gamma)\neq (0, 0, 0)$, and for $q(t):=p(t)+\eta(t)$ with $\eta(t):=
\displaystyle\int_{[t_0, t)}\nu(s)d\mu(s)$ if $t\in [t_0, T)$ and $\eta(T):=\displaystyle\int_{[t_0, T]}\nu(s)d\mu(s)$, the following holds true:
\begin{enumerate}[\rm (i)]
\item $\nu(t)\in \partial^>_x h(t, \bar x(t))\ \mu-\mbox{a.e.};$
 \item $-\dot p(t)\in \co \partial_x\mathcal{H}(t, \bar x(t), q(t), \bar u(t))$ a.e.;
\item $(p(t_0), -q(T))\in \gamma \partial g(\bar x(t_0), \bar x(T))+N_C(\bar x(t_0), \bar x(T))$;
\item $\mathcal{H}(t, \bar x(t), q(t), \bar u(t))=\max_{u\in U(t)}\mathcal{H}(t, \bar x(t), q(t), u)$ a.e.
\end{enumerate}
\end{theorem}

Applying Theorem \ref{V_thm9.3.1 necessary condition} to unconstrained optimal control problems, one has next proposition.

\begin{proposition}[{See \cite[Theorem~6.2.1]{Vinter_2000}}]\label{V_thm6.2.1 necessary condition}
	Suppose that $\mathcal M$ is an optimal control problem without state constraints. Let $(\bar x, \bar u)$ be a $W^{1,1}$ local minimizer for $\mathcal M$. Assume that for some $\delta>0$, the following hypotheses are satisfied.
	\begin{enumerate}[\rm (H1)]
		\item For every $x\in\R^n$, the function $f(., x, .): [t_0, T]\times \R^m \to \R^n$ is $\mathcal{L}\times \mathcal{B}^m$ measurable. In addition, there exists a Borel measurable function $k:[t_0, T]\times \R^m \to \R$ such that $t \mapsto k(t, \bar u(t))$ is integrable and 
		\begin{equation*}
		\| f(t, x, u)-f(t, x', u)\|\leq k(t, u)\|x-x'\|, \quad \forall x, x'\in \bar x(t)+\delta\bar B, u \in U(t), a.e.;
		\end{equation*}
		\item $\gph U$ is an $\mathcal{L}\times \mathcal{B}^m$ measurable set in $[t_0,T]\times\R^m$;
		\item $g$ is locally Lipschitz continuous.
	\end{enumerate}
	Then there exist $p\in W^{1,1}([t_0, T]; \R^n)$ and  $\gamma \geq 0$ such that $(p, \gamma)\neq (0, 0)$ and the following holds true:
	\begin{enumerate}[\rm (i)]
		\item $-\dot p(t)\in \co \partial_x\mathcal{H}(t, \bar x(t), p(t), \bar u(t))$ a.e.;
		\item $(p(t_0), -p(T))\in \gamma \partial g(\bar x(t_0), \bar x(T))+N_C(\bar x(t_0), \bar x(T))$;
		\item $\mathcal{H}(t, \bar x(t), p(t), \bar u(t))=\max_{u\in U(t)}\mathcal{H}(t, \bar x(t), p(t), u)$.
	\end{enumerate}
\end{proposition}

\subsection{Solution Existence in State Constrained Optimal Control}
To recall a solution existence theorem for optimal control problems with state constraints of the Mayer type, we will use the notations and concepts given in \cite[Section~9.2]{Cesari_1983}. Let $A$ be a subset of $\R\times \R^n$ and $U: A\rightrightarrows \R^m$ be a set-valued map defined on $A$. Let $$M:=\{(t, x, u)\in \R\times\R^n\times \R^m \;:\; (t, x)\in A,\ u \in U(t, x)\},$$ and $f=(f_1, f_2, \dots, f_n): M \to \R^n$ be a single-valued map defined on $M$. Let $B$ be a given subset of $\R\times \R^n\times\R\times \R^n$ and $g: B \to \R$ be a real function defined on $B$. Consider the optimal control problem of the Mayer type
\begin{equation}\label{cost functional_SET}
\mbox{Minimize}\ \; g(t_0, x(t_0), T,  x(T))
\end{equation} over $x \in W^{1,1}([t_0, T]; \R^n)$  and measurable functions $u:[t_0, T]~\to~\R^m$ satisfying
\begin{equation}\label{state control system_SET}
	\begin{cases}
		\dot x(t)=f(t, x(t), u(t)),\quad &\mbox{a.e.\ } t\in [t_0, T]\\
		(t, x(t))\in A, & \mbox{for all }t\in [t_0, T]\\
		(t_0, x(t_0), T, x(T))\in B\\
		u(t)\in U(t, x(t)), &\mbox{a.e.\ } t\in [t_0, T],
	\end{cases}
\end{equation}
where $[t_0, T]$ is a given interval. The problem \eqref{cost functional_SET}--\eqref{state control system_SET} will be denoted by $\mathcal{M}_1$.

A \textit{feasible process} for $\mathcal M_1$ is a pair of functions $(x, u)$ with $x: [t_0, T] \to \R^n$ being absolutely continuous on $[t_0, T]$, $u:[t_0, T] \to \R^m$ being measurable, such that  all the requirements in \eqref{state control system_SET} are satisfied. If  $(x, u)$ is a feasible process for $\mathcal M_1$, then $x$
is said to be a \textit{feasible trajectory}, and $u$ a \textit{feasible control function} for $\mathcal M_1$. The set of all feasible processes for $\mathcal M_1$ is denoted by $\Omega$.

Let $A_0=\big\{t\in\mathbb R\,:\, \exists x\in \mathbb R^n\ {\rm s.t.}\ (t,x)\in A\big\}$, i.e., $A_0$ is the projection of $A$ on the $t-$axis. Set $$A(t)=\big\{x\in \R^n \;:\; (t, x)\in A\big\}\quad\; (t\in A_0)$$ and  $$Q(t, x)=\big\{z\in \R^n \;:\; z=f(t, x, u),\ u \in U(t, x)\big\} \quad\; ((t, x)\in A).$$

The forthcoming statement is called \textit{Filippov's Existence Theorem for Mayer problems}.
\begin{theorem}[{see \cite[Theorem~9.2.i and Section~9.4]{Cesari_1983}}]\label{Filippov's_Existence_Thm}
Suppose that $\Omega$ is nonempty, $B$ is closed, $g$ is lower semicontinuous on $B$, $f$ is continuous on $M$ and, for almost every $t\in [t_0, T]$, the sets $Q(t, x)$, $x\in A(t)$, are convex. Moreover,  assume either that $A$ and $M$ are compact or that $A$ is not compact but closed and the following three conditions hold
\begin{enumerate}[\rm (a)]
\item For any $\varepsilon\geq 0$, the set $M_\varepsilon:=\{(t, x, u)\in M \;:\; \|x\| \leq \varepsilon\}$ is compact;
\item There is a compact subset $P$ of $A$ such that every feasible trajectory $x$ of $\mathcal M_1$ passes through at least one point of $P$;
\item There exists $c\geq 0$ such that 
\begin{equation*}
x_1 f_1(t, x, u) + x_2 f_2(t, x, u)+\dots+x_n f_n(t, x, u) \leq c (\|x\|^2+1)\quad\; \forall (t, x, u)\in M.
\end{equation*}
\end{enumerate}
Then, $\mathcal M_1$ has a $W^{1, 1}$ global minimizer.
\end{theorem}
Clearly, condition (b) is satisfied if the initial point $(t_0, x(t_0))$ or the end point $(T, x(T))$ is fixed. As shown in \cite[p.~317]{Cesari_1983}, the following condition implies (c): 
\begin{enumerate}[\rm ($c_0$)]
\item \textit{There exists $c\geq 0$ such that $
\|f(t, x, u)\| \leq c (\|x\|+1)$ for all $(t, x, u)\in M$.}
\end{enumerate}

\section{Control Problems without State Constraints}\label{Example 1}

Denote by $(FP_1)$ the finite horizon optimal control problem of the Lagrange type
\begin{equation} \label{cost functional_FP_1}
\mbox{Minimize}\ \; J(x,u)=\int_{t_0}^{T} \big[-e^{-\lambda t}(x(t)+u(t))\big] dt
\end{equation}
over $x \in W^{1,1}([t_0, T], \R)$  and measurable function $u:[t_0, T] \to \R$ satisfying
\begin{equation} \label{state control system_FP_1}
\begin{cases}
\dot x(t)=-au(t),\quad &\mbox{a.e.\ } t\in [t_0, T]\\
x(t_0)=x_0\\
u(t)\in [-1, 1], &\mbox{a.e.\ } t\in [t_0, T],\\
\end{cases}
\end{equation}
with $a > \lambda>0$, $T>t_0\geq 0$, and  $x_0 \in \R$ being given.

\medskip
To treat  $(FP_1)$  in \eqref{cost functional_FP_1}--\eqref{state control system_FP_1} as a problem of the Mayer type, we set $x(t)=(x_1(t), x_2(t))$, where $x_1(t)$ plays the role of the state variable $x(t)$ in $(FP_1)$, and $$x_2(t):= \int_{t_0}^{t} \big[-e^{-\lambda t}(x_1(\tau)+u(\tau))\big] d\tau$$ for all $t\in [0, T]$. Then $(FP_1)$ is equivalent to the problem
\begin{equation} \label{cost functional_FP_1a}
\mbox{Minimize}\ \; x_2(T)
\end{equation}
over $x=(x_1, x_2) \in W^{1,1}([t_0, T], \R^2)$ and measurable functions $u:[t_0, T] \to \R$ satisfying
\begin{equation}\label{state control system_FP_1a}
\begin{cases}
\dot x_1(t)=-au(t),\quad &\mbox{a.e.\ } t\in [t_0, T]\\
\dot x_2(t)=-e^{-\lambda t}(x_1(t)+u(t)), &\mbox{a.e.\ } t\in [t_0, T]\\
(x(t_0), x(T))\in \{(x_0, 0)\}\times \R^2\\
u(t)\in [-1, 1], &\mbox{a.e.\ } t\in [t_0, T].\\
\end{cases}
\end{equation}
The problem \eqref{cost functional_FP_1a}--\eqref{state control system_FP_1a} is abbreviated to  $(FP_{1a})$. 

\subsection{Solution Existence}\label{SE_without_S_constraint}
Clearly, $(FP_{1a})$ is of the form $\mathcal{M}_1$ (see Subsection 2.3) with $n=2$, $m=1$, $A=[t_0, T]~\times~\R^2$, $U(t, x)=[-1, 1]$ for all $(t, x)\in A$, $B=\{t_0\}\times\{(x_0, 0)\}\times \R\times \R^2$, $g(t_0, x(t_0), T, x(T))=x_2(T)$, $M=A\times [-1, 1]$, $f(t, x, u)=(-au, -e^{-\lambda t}(x_1+u))$ for all $(t, x, u)\in M$. We are going to show that $(FP_{1a})$ satisfies all the assumptions of Theorem~\ref{Filippov's_Existence_Thm}.

Clearly, the pair $(x, u)$, where $u(t)=0$, $x_1(t)=x_0$, and $x_2(t)= -x_0\displaystyle\int_{t_0}^{t}e^{-\lambda \tau}d\tau$ for all $t\in [t_0, T]$, is a feasible process for $(FP_{1a})$. Thus, the set $\Omega$ of feasible processes is nonempty. Besides, $B$ is closed, $g$ is lower semicontinuous on $B$, $f$ is continuous on $M$. Moreover, by the formula for $A$, one has  $A_0=[t_0, T]$ and $A(t)=\R^2$ for all $t\in A_0$. In addition, from the formulas for $M$, $U$, and $f$, one gets
\begin{align*}
Q(t, x) &=\big\{z\in \R^2 \;:\; z=f(t, x, u),\ u \in U(t, x)\big\}\\
&=\big\{z\in \R^2 \;:\; z=(-au, -e^{-\lambda t}(x_1+u)),\ u \in [-1, 1]\big\}\\
&=\big\{(0, -e^{-\lambda t}x_1)\}+\{(-a, -e^{-\lambda t})u \;:\; u\in [-1, 1]\big\}
\end{align*}
for any $(t, x)\in A$. Thus, for every $t \in [t_0, T]$, the sets $Q(t, x)$, $x\in A(t)$, are line segments; hence they are convex. Since $A$ is closed, but not compact, we have to check the conditions (a)--(c) in Theorem~\ref{Filippov's_Existence_Thm}.

\textit{Condition} (a): For any $\varepsilon\geq 0$, since 
\begin{align*}
M_\varepsilon&=\{(t, x, u)\in M \;:\; \|x\| \leq \varepsilon\}\\
&=\{(t, x, u)\in[t_0, T]\times \R^2\times [-1,1] \;:\; \|x\| \leq \varepsilon\}\\
&= [t_0, T]\times \{x\in \R^2\;:\;\|x\| \leq \varepsilon\}\times [-1,1],
\end{align*}
one sees that $M_\varepsilon$ is compact.

\textit{Condition} (b): Obviously, $P:=\{t_0\}\times\{(x_0, 0)\}$ is a compact subset of $A$, and every feasible trajectory passes through the unique point of $P$. Thus, condition (b) is fulfilled. 

\textit{Condition} (c): Choosing $c=a+1$, we have
\begin{align*}
\|f(t, x, u)\|=\|(-au, -e^{-\lambda t}(x_1+u))& \leq a|u|+e^{-\lambda t}|x_1+u|\\
&\leq a+|x_1|+1\\ & \leq c(\|x\|+1)
\end{align*}
 for any $(t, x, u)\in M$, because $u\in [-1,1]$ and $e^{-\lambda t}\leq 1$ for $t\geq t_0 \geq 0$. Thus, condition ($c_0$), which implies (c), is satisfied.

By Theorem~\ref{Filippov's_Existence_Thm}, $(FP_{1a})$ has a $W^{1, 1}$ global minimizer. Therefore, $(FP_{1})$ has a $W^{1, 1}$ global minimizer by the equivalence of $(FP_{1a})$ and $(FP_{1})$.

\subsection{Necessary Optimality Conditions} 
To obtain necessary conditions  for $(FP_{1a})$, we note that $(FP_{1a})$ is in the form of $\mathcal{M}$ with  $g(x, y)=y_2$, $f(t, x, u)=(-au, -e^{-\lambda t}(x_1+u))$, $C=\{(x_0, 0)\}\times \R^2$, $U(t)=[-1, 1]$, and $h(t, x)=0$ for all $x=(x_1, x_2) \in \R^2$, $y=(y_1, y_2) \in \R^2$, $t\in [t_0, T]$, and $u\in \R$. Since $(FP_{1a})$ is an optimal control problem without state constraints, we can apply both  Proposition~\ref{V_thm6.2.1 necessary condition} Theorem~\ref{V_thm9.3.1 necessary condition} to this problem. In accordance with \eqref{Hamiltonian}, the Hamiltonian of $(FP_{1a})$ is given by 
\begin{equation}\label{Hamiltonian-FP1a}
\mathcal{H}(t, x, p, u)=-aup_1-e^{-\lambda t}(x_1+u)p_2 \quad \forall(t, x, p, u)\in [t_0, T]\times \R^2\times \R^2 \times \R,
\end{equation}
while by \eqref{h-partial subdiff} we have $\partial^>_x h(t, x)=\emptyset$ for all $(t, x)\in [t_0, T]\times \R^2$. Let $(\bar x, \bar u)$ be a $W^{1,1}$ local minimizer of $(FP_{1a})$. 
\subsubsection{Necessary Optimality Conditions for  $(FP_{1a})$ in Terms of Proposition~\ref{V_thm6.2.1 necessary condition}}\label{Sub1-FP1a}
It is clear that the assumptions (H1)--(H3) of Proposition~\ref{V_thm6.2.1 necessary condition} are satisfied for $(FP_{1a})$. So, there exist $p\in W^{1,1}([t_0, T]; \R^2)$ and $\gamma \geq 0$ such that $(p, \gamma)\neq (0, 0)$, and conditions~(i)--(iii) of Proposition~\ref{V_thm6.2.1 necessary condition} hold true. Let us analyze these conditions.

\textbf{Condition (i)}: By \eqref{Hamiltonian-FP1a}, $\mathcal{H}$ is differentiable in $x$ and $\partial_x\mathcal{H}(t, x, p, u)=\{(-e^{-\lambda t}p_2, 0)\}$ for all $(t, x, p, u)\in [t_0, T]\times \R^2\times \R^2 \times \R$. Thus, condition (i) implies that $\dot p_1(t) = e^{-\lambda t}p_2(t)$ for a.e. $t\in [t_0, T]$ and $p_2(t)$ is a constant function.

\textbf{Condition (ii)}: By the formulas for $g$ and $C$, we have $\partial g(\bar x(t_0), \bar x(T))=\{(0, 0, 0, 1)\}$ and $N_C(\bar x(t_0), \bar x(T))=\R^2\times\{(0, 0)\}$. Thus, condition (ii) implies that
$$(p(t_0), -p(T))\in \gamma \{(0, 0, 0, 1)\}+\R^2\times\{(0, 0)\};$$
hence $p_1(T)=0$ and $p_2(T)=-\gamma$. As $p_2(t)$ is a constant function, we have  $p_2(t)=-\gamma$ for all $t\in [t_0, T]$. So, the above analysis of condition (i) gives $p_1(t)=\dfrac{\gamma}{\lambda}\big(e^{-\lambda t}-e^{-\lambda T}\big)$ for all $t\in [t_0, T]$. Since  $(p, \gamma)\neq (0, 0)$, we must have $\gamma > 0$.

\textbf{Condition (iii)}: Due to \eqref{Hamiltonian-FP1a}, condition (iii) means that
\begin{equation*}
-a\bar u(t)p_1(t)-e^{-\lambda t}[x_1(t)+\bar u(t)]p_2(t)=\max_{u\in [-1, 1]}\left\{-aup_1(t)-e^{-\lambda t}[x_1(t)+u]p_2(t) \right\}
\end{equation*}
for a.e. $t\in [t_0, T]$.
Equivalently,
\begin{equation}\label{min_condition}
[ap_1(t)+e^{-\lambda t}p_2(t)]\bar u(t)=\min_{u\in [-1, 1]}\left\{[ap_1(t)+e^{-\lambda t}p_2(t)]u\right\} \quad \mbox{a.e. } t\in [t_0, T].
\end{equation}
Setting $\varphi(t):=ap_1(t)+e^{-\lambda t}p_2(t)$ for $t\in [t_0, T]$, we have 
\begin{align*}
\varphi(t) =a\dfrac{\gamma}{\lambda}\big(e^{-\lambda t}-e^{-\lambda T}\big)-\gamma e^{-\lambda t}=\gamma (\dfrac{a}{\lambda}-1)e^{-\lambda t}-\gamma\dfrac{a}{\lambda} e^{-\lambda T}
\end{align*}
for a.e. $t\in [t_0, T]$. As $\dfrac{a}{\lambda}>1$, we see that $\varphi$  is decreasing on $\R$. In addition, it is clear that $\varphi(T)=-\gamma e^{-\lambda T}<0$, and $\varphi(t)=0$ if and only if $t=\bar t$, where $\bar t:=T-\dfrac{1}{\lambda}\ln \dfrac{a}{a-\lambda}$. 

We have the following cases.

{\sc Case  A:} \textit{$t_0\geq \bar t$}. Then $\varphi(t)<0$ for all $t\in (t_0, T]$. Therefore, condition \eqref{min_condition} implies $\bar u(t)=1$ for all $t\in [t_0, T]$. Hence, by \eqref{state control system_FP_1a}, $\bar x_1(t)=x_0-a(t-t_0)$ for a.e. $t\in [t_0, T]$.

{\sc Case  B:} \textit{$t_0 < \bar t$}. Then $\varphi(t)>0$ for $t\in [t_0, \bar t)$ and $\varphi(t)<0$ for $t\in (\bar t, T]$. Thus, \eqref{min_condition} yields $\bar u(t)=-1$ for $t\in [t_0, \bar t)$ and $\bar u(t)=1$ for a.e. $t\in (\bar t, T]$; hence $\bar x_1(t)=x_0+a(t-t_0)$ for every $t\in [t_0, \bar t]$ and $\bar x_1(t)=x_0-a(t+t_0-2\bar t)$ for every $t\in (\bar t, T]$.

\subsubsection{Necessary Optimality Conditions for  $(FP_{1a})$ in Terms of Theorem~\ref{V_thm9.3.1 necessary condition}}
Since the  assumptions (H1)--(H4) of Theorem~\ref{V_thm9.3.1 necessary condition} are satisfied for $(FP_{1a})$, by that theorem one can find $p\in W^{1,1}([t_0, T]; \R^2)$, $\gamma \geq 0$, $\mu \in C^\oplus(t_0, T)$, and a Borel measurable  function $\nu:[t_0, T]\to  \R^2$ such that $(p, \mu, \gamma)\neq (0, 0, 0)$, and for $q(t):=p(t)+\eta(t)$ with $\eta:[t_0, T]\to  \R^2$ being given by $\eta(t):=
\displaystyle\int_{[t_0, t)}\nu(s)d\mu(s)$ if $t\in [t_0, T)$ and $\eta(T):=\displaystyle\int_{[t_0, T]}\nu(s)d\mu(s)$, conditions (i)--(iv) in Theorem~\ref{V_thm9.3.1 necessary condition} hold true. Since $\partial^>_x h(t, \bar x(t))=\emptyset$ for all $t\in [t_0, T]$, the inclusion $\nu(t)\in \partial^>_x h(t, \bar x(t))$ is violated at every $t\in [t_0,T]$. Hence, condition (i) forces $\mu=0$. We see that condition (iv) is fulfilled and the conditions (ii)--(iv) in  Theorem~\ref{V_thm9.3.1 necessary condition} recover the conditions (i)--(iii) of Proposition~\ref{V_thm6.2.1 necessary condition}.

\medskip
	Going back to the original problem $(FP_1)$, we can put the obtained results in the following theorem.

\begin{Theorem}\label{Thm1} Given any $a,\lambda$ with $a>\lambda>0$, define $\rho=\dfrac{1}{\lambda}\ln \dfrac{a}{a-\lambda}>0$ and  $\bar t=T-\rho$. Then, problem $(FP_1)$ has a unique local solution $(\bar x,\bar u)$, which is a global solution, where $\bar u(t)=-a^{-1}\dot{\bar x}(t)$ for almost everywhere $t\in [t_0, T]$ and $\bar x(t)$ can be described as follows:
	\begin{description}
		\item{\rm (a)} If $t_0 \geq \bar t$ (i.e., $T-t_0 \leq \rho$), then 
		\begin{equation*}
		\bar x(t)=x_0-a(t-t_0), \quad t \in [t_0, T].
		\end{equation*}
		\item{\rm (b)} If $t_0 < \bar t$ (i.e., $T-t_0 > \rho$), then 
		\begin{equation*}
	    \bar x(t)=
	    \begin{cases}
	    x_0+a(t-t_0), \quad & t \in [t_0, \bar t]\\
	    x_0-a(t+t_0-2\bar t), & t \in (\bar t, T].
	    \end{cases}
		\end{equation*}
	\end{description}
\end{Theorem}
\begin{proof} The assertions (a) and (b) are straightforward from the results obtained in Case~A and Case~B of Subsection~\ref{Sub1-FP1a}, because $\bar x_1(t)$ in $(FP_{1a})$ coincides with $\bar x(t)$ in $(FP_1)$.
	  	\end{proof}
	  		  	
\section{Control Problems with Unilateral Constraints}\label{Example 2}

By $(FP_2)$ we denote the finite horizon optimal control problem of the Lagrange type
		\begin{equation} \label{cost functional_FP_2}
		\mbox{Minimize}\ \; J(x,u)=\int_{t_0}^{T} \big[-e^{-\lambda t}(x(t)+u(t))\big] dt
		\end{equation}
		over $x \in W^{1,1}([t_0, T], \R)$  and measurable functions $u:[t_0, T] \to\R$ satisfying
		\begin{equation} \label{state control system_FP_2}
		\begin{cases}
		\dot x(t)=-au(t),\quad &\mbox{a.e.\ } t\in [t_0, T]\\
		x(t_0)=x_0\\
		u(t)\in [-1, 1], &\mbox{a.e.\ } t\in [t_0, T]\\
		x(t)\leq 1, & \forall t\in [t_0, T]
		\end{cases}
		\end{equation}
		with $a > \lambda>0$, $T>t_0\geq 0$, and  $x_0 \leq 1$ being given.
		
		\medskip
		We transform this problem into one of the Mayer type by setting $x(t)=(x_1(t), x_2(t))$, where $x_1(t)$ plays the role of $x(t)$ in \eqref{cost functional_FP_2}--\eqref{state control system_FP_2} and \begin{equation}\label{phase_second_component}x_2(t):= \int_{t_0}^{t} \big[-e^{-\lambda\tau}(x_1(\tau)+u(\tau))\big] d\tau\end{equation} for all $t\in [0, T]$. Thus, $(FP_2)$ is equivalent to the problem
		\begin{equation}\label{cost functional_FP_2a}
		\mbox{Minimize}\ \; x_2(T)
		\end{equation}
		over $x=(x_1, x_2) \in W^{1,1}([t_0, T], \R^2)$  and measurable functions $u:[t_0, T] \to\R$ satisfying
		\begin{equation} \label{state control system_FP_2a}
		\begin{cases}
		\dot x_1(t)=-au(t),\quad &\mbox{a.e.\ } t\in [t_0, T]\\
		\dot x_2(t)=-e^{-\lambda t}(x_1(t)+u(t)), &\mbox{a.e.\ } t\in [t_0, T]\\
		(x(t_0), x(T))\in \{(x_0, 0)\}\times \R^2\\
		u(t)\in [-1, 1], &\mbox{a.e.\ } t\in [t_0, T]\\
		x_1(t)\leq 1, &\forall t\in [t_0, T].
		\end{cases}
		\end{equation} We denote problem \eqref{cost functional_FP_2a}--\eqref{state control system_FP_2a} by $(FP_{2a})$.
		
	\subsection{Solution Existence}\label{SE_1-sided_S_constraint}
	To check that $(FP_{2a})$ is of the form $\mathcal{M}_1$ (see Subsection 2.3), we choose  $n=2$, $m=1$, $A=[t_0, T]~\times(-\infty, 1]\times \R$, $U(t, x)=[-1, 1]$ for all $(t, x)\in A$, $B=\{t_0\}\times\{(x_0, 0)\}\times \R\times \R^2$, $g(t_0, x(t_0), T, x(T))=x_2(T)$, $M=A\times [-1, 1]$, $f(t, x, u)=(-au, -e^{-\lambda t}(x_1+u))$ for all $(t, x, u)\in M$. In comparison with the problem $(FP_{1a})$, the only change in this formulation of $(FP_{2a})$ is that we have $A=[t_0, T]~\times(-\infty, 1]\times \R$ instead of $A=[t_0, T]~\times \R^2$. Thus, to show that $(FP_{2a})$ satisfies all the assumptions of Theorem~\ref{Filippov's_Existence_Thm}, we can use the arguments in Subsection~\ref{SE_without_S_constraint}, except those related to the convexity of the sets $Q(t,x)$ and the compactness of $M_\varepsilon$, which have to be verified in a slightly different manner. 
		
		By the above formula for $A$, we have  $A_0=[t_0, T]$ and $A(t)=(-\infty, 1]\times \R$ for all $t\in A_0$. As in Subsection~\ref{SE_without_S_constraint}, we have \begin{align*}
		Q(t, x)=\big\{(0, -e^{-\lambda t}x_1)\}+\{(-a, -e^{-\lambda t})u \;:\; u\in [-1, 1]\big\}
		\end{align*}
		for any $(t, x)\in A$. Thus, the assumption of Theorem~\ref{Filippov's_Existence_Thm} on the convexity of the sets  $Q(t, x)$, $x\in A(t)$, for almost every $t\in [t_0, T]$, is satisfied. Since $M=[t_0, T]~\times(-\infty, 1]\times \R\times [-1, 1]$, for any $\varepsilon\geq 0$, one has
		\begin{align*}
		M_\varepsilon&=\{(t, x, u)\in M \;:\; \|x\| \leq \varepsilon\}\\
		&=\{(t, x, u)\in[t_0, T]~\times(-\infty, 1]\times \R\times [-1, 1] \;:\; \|x\| \leq \varepsilon\}.
		\end{align*}
		As $M_\varepsilon$ is closed and contained in the compact set $[t_0, T]\times \{x\in \R^2\;:\;\|x\| \leq \varepsilon\}\times [-1,1]$, it is compact.
		
		It follows from Theorem~\ref{Filippov's_Existence_Thm} that $(FP_{2a})$ has a $W^{1, 1}$ global minimizer. Therefore, by the equivalence of~$(FP_{2})$ and $(FP_{2a})$, we can assert that $(FP_{2})$ has a $W^{1, 1}$ global minimizer.
		
		\subsection{Necessary Optimality Conditions}\label{Subsection4.2} 
		
	In order to apply Theorem~\ref{V_thm9.3.1 necessary condition} for solving $(FP_2)$, we observe that $(FP_{2a})$ is in the form of~$\mathcal M$ with  $g(x, y)=y_2$, $f(t, x, u)=(-au, -e^{-\lambda t}(x_1+u)),$ $C=\{(x_0, 0)\}\times \R^2$, $U(t)=[-1, 1]$, and $h(t, x)=x_1-1$ for all $t\in [t_0, T]$, $x=(x_1, x_2) \in \R^2$, $y=(y_1, y_2) \in \R^2$ and $u\in \R$. 
		
		\medskip
		The forthcoming two propositions describe a fundamental properties of the local minimizers of the problem $(FP_{2a})$, which is obtained from the optimal control problem of the Lagrange type $(FP_{2})$ by introducing the artificial variable $x_2$. Similar statements as those in the first proposition are valid for any optimal control problem of the Mayer type, which is  obtained from an optimal control problem of the Lagrange type in the same manner. While, the claims in the second proposition hold true for every optimal control problem of the Mayer type, whose objective function does not depend on the initial point.
		
		\begin{Proposition}\label{lemma_basic_property_1}
		Suppose that $(\bar x, \bar u)$ is a $W^{1,1}$ local minimizer for $(FP_{2a})$. Then, for any $\tau_1,\tau_2\in [t_0,T]$ with $\tau_1<\tau_2$, the restriction of $(\bar x, \bar u)$ on $[\tau_1,\tau_2]$, i.e., the process $(\bar x(t), \bar u(t))$ with $t\in [\tau_1,\tau_2]$, is  a $W^{1,1}$ local minimizer for the following  optimal control problem of the Mayer type
		\begin{equation*}\label{cost functional_FP_2aR}
	    {\rm Minimize}\ \; x_2(\tau_2)
		\end{equation*}
		{\rm over $x=(x_1, x_2) \in W^{1,1}([\tau_1, \tau_2], \R^2)$  and measurable functions $u:[\tau_1, \tau_2] \to\R$ satisfying
		\begin{equation*}\label{state control system_FP_2aR}
		\begin{cases}
		\dot x_1(t)=-au(t),\quad &\mbox{a.e.\ } t\in [\tau_1,\tau_2]\\
		\dot x_2(t)=-e^{-\lambda t}(x_1(t)+u(t)),&\mbox{a.e.\ } t\in [\tau_1,\tau_2]\\
		(x(\tau_1), x(\tau_2))\in \{(\bar x_1(\tau_1),\bar x_2(\tau_1))\}\times \{\bar x_1(\tau_2)\}\times \R\\
		u(t)\in [-1, 1], &\mbox{a.e.\ } t\in [\tau_1, \tau_2]\\
		x_1(t)\leq 1, & \forall t\in [\tau_1, \tau_2],
		\end{cases}
		\end{equation*}}which is denoted by $(FP_{2a})|_{[\tau_1, \tau_2]}$. In another words, for any $\tau_1,\tau_2\in [t_0,T]$ with $\tau_1<\tau_2$, the restriction of a $W^{1,1}$ local minimizer for $(FP_{2a})$ on the time segment $[\tau_1,\tau_2]$ is a $W^{1,1}$ local minimizer for the Mayer problem $(FP_{2a})|_{[\tau_1, \tau_2]}$, which is obtained from $(FP_{2a})$ by replacing $t_0$ with $\tau_1$, $T$ with $\tau_2$, and $C$ with $\widetilde C:=\{(\bar x_1(\tau_1),\bar x_2(\tau_1))\}\times \{\bar x_1(\tau_2)\}\times \R$.
		\end{Proposition}
		\begin{proof}
			Since $(\bar x, \bar u)$ is a $W^{1,1}$ local minimizer for $(FP_{2a})$, by Definition \ref{local_minimizer} there exists $\delta>0$ such that the process $(\bar x, \bar u)$ minimizes the quantity $g(x(t_0), x(T))=x_2(T)$ over all feasible processes $(x, u)$ of $(FP_{2a})$ with $\|\bar x-x\|_{W^{1,1}} \leq \delta$. 
			
			Clearly, the restriction of $(\bar x, \bar u)$ on $[\tau_1,\tau_2]$ satisfies the conditions given in \eqref{cost functional_FP_2aR}. Thus, it is a feasible process for  $(FP_{2a})|_{[\tau_1, \tau_2]}$.
			
			Let $(x(t), u(t))$, $t\in [\tau_1, \tau_2]$, be an arbitrary feasible process of $(FP_{2a})|_{[\tau_1, \tau_2]}$ satisfying $$\|\bar x-x\|_{W^{1,1}([\tau_1, \tau_2],\R^2)} \leq\delta.$$ Consider the pair of functions $(\widetilde x,\widetilde u)$, where $\widetilde x=(\widetilde x_1,\widetilde x_2)$, which is given by
			\begin{align*}
			\widetilde x_1(t):=
			\begin{cases}
			\bar x_1(t), \   & t\in [t_0,\tau_1]\cup [\tau_2,T]\\
			x_1(t),  & t\in (\tau_1, \tau_2),
			\end{cases}
				\end{align*}
				\begin{align*}
				\widetilde x_2(t):=
				\begin{cases}
				\bar x_2(t), \  & t\in [t_0,\tau_1]\\
				x_2(t), & t\in (\tau_1, \tau_2)\\
					x_2(\tau_2)+\displaystyle\int_{\tau_2}^{t} \big[-e^{-\lambda\tau}(\bar x_1(\tau)+\bar u(\tau))\big] d\tau , & t\in [\tau_2,T],
				\end{cases}
				\end{align*}
				and
					\begin{align*}
			\widetilde u(t):=
			\begin{cases}
			\bar u(t), \  & t\in [t_0,\tau_1]\cup [\tau_2,T]\\
				u(t),	& t\in (\tau_1, \tau_2).
			\end{cases}
			\end{align*}
			Clearly, $(\widetilde x,\widetilde u)$ is a feasible process of $(FP_{2a})$ satisfying  $\|\bar x-\widetilde x\|_{W^{1,1}([t_0, T],\R^2)} \leq \delta$. Thus, one must have $g(\widetilde x(T))\geq g(\bar x(T))$ or, equivalently,
			\begin{eqnarray*}
		 x_2(\tau_2)+\displaystyle\int_{\tau_2}^{T} \omega(\tau) d\tau 
			\geq \bar x(\tau_2)  + \displaystyle\int_{\tau_2}^{T} \omega(\tau) d\tau  ,
			\end{eqnarray*}	where $\omega(\tau):=-e^{-\lambda\tau}(\bar x_1(\tau)+\bar u(\tau))$. Hence, one obtains the inequality $x_2(\tau_2)\geq \bar x(\tau_2)$ proving that  the restriction of $(\bar x, \bar u)$ on $[\tau_1,\tau_2]$ is  a $W^{1,1}$  local minimizer for $(FP_{2a})|_{[\tau_1, \tau_2]}$.
		\end{proof}
		\begin{Proposition}\label{lemma_basic_property_2}
		Suppose that $(\bar x, \bar u)$ is a $W^{1,1}$ local minimizer for $(FP_{2a})$. Then, for any $\tau_1\in [t_0,T)$, the restriction of the process $(\bar x, \bar u)$ on the time segment $[\tau_1, T]$, i.e., the process $(\bar x(t), \bar u(t))$ with $t\in [\tau_1, T]$, is  a $W^{1,1}$ local minimizer for the following  optimal control problem of the Mayer type
				\begin{equation*}\label{cost functional_FP_2b}
			    {\rm Minimize}\ \; x_2(T)
				\end{equation*}
				{\rm over $x=(x_1, x_2) \in W^{1,1}([\tau_1, T], \R^2)$  and measurable functions $u:[\tau_1, T] \to\R$ satisfying
				\begin{equation*}\label{state control system_FP_2b}
				\begin{cases}
				\dot x_1(t)=-au(t),\quad &\mbox{a.e.\ } t\in [\tau_1,T]\\
				\dot x_2(t)=-e^{-\lambda t}(x_1(t)+u(t)),&\mbox{a.e.\ } t\in [\tau_1,T]\\
				(x(\tau_1), x(T))\in \{(\bar x_1(\tau_1),\bar x_2(\tau_1))\}\times \R^2\\
				u(t)\in [-1, 1], &\mbox{a.e.\ } t\in [\tau_1, T]\\
				x_1(t)\leq 1, & \forall t\in [\tau_1, T],
				\end{cases}
				\end{equation*}}which is denoted by $(FP_{2b})$. In another words, for any $\tau_1\in [t_0,T)$, the restriction of a $W^{1,1}$ local minimizer for $(FP_{2a})$ on the time segment $[\tau_1, T]$	is a $W^{1,1}$ local minimizer for the Mayer problem $(FP_{2b})$, which is obtained from $(FP_{2a})$ by replacing $t_0$ with $\tau_1$.
		\end{Proposition}
		\begin{proof} For a fixed $\tau_1\in [t_0,T)$, let $(FP_{2b})$ be defined as in the formulation of the lemma. It is clear that the process $(\bar x(t), \bar u(t))$, $t\in [\tau_1, T]$, is feasible for $(FP_{2b})$. Since $(\bar x, \bar u)$ is a $W^{1,1}$ local minimizer of $(FP_{2a})$, by Definition \ref{local_minimizer} there exists $\delta>0$ such that the process $(\bar x, \bar u)$ minimizes the quantity $g(x(t_0), x(T))=x_2(T)$ over all feasible processes $(x, u)$ of $(FP_{2a})$ with $\|\bar x-x\|_{W^{1,1}} \leq \delta$. Let $(x(t), u(t))$, $t\in [\tau_1, T]$, be an arbitrary feasible process of $(FP_{2b})$ satisfying $\|\bar x-x\|_{W^{1,1}([\tau_1, T])} \leq \delta$. Consider the pair of functions $(\widetilde x,\widetilde u)$ given by
		 				\begin{align*}
		 					\widetilde x(t):=
		 				\begin{cases}
		 			 \bar x(t), \quad t\in [t_0, \tau_1)\\
		 				 x(t), \quad t\in [\tau_1, T]
		 				\end{cases}
		 				\quad\; \mbox{and} \quad\;
		 				\widetilde u(t):=
		 				\begin{cases}
		 				 \bar u(t), \quad t\in [t_0, \tau_1)\\
		 				 u(t), \quad t\in [\tau_1, T].
		 				\end{cases}
		 				\end{align*}
		 		Clearly, $(\widetilde x,\widetilde u)$ is a feasible process of $(FP_{2a})$ satisfying  $\|\bar x-\widetilde x\|_{W^{1,1}([t_0, T])} \leq \delta$. Thus, one must have $g(\widetilde x(T))\geq g(\bar x(T))$. Since $\widetilde x(T)=x(T) $, one obtains the inequality $g(x(T))\geq g(\bar x(T))$, which justifies the assertion of the proposition.  
		\end{proof}
	 		
		In accordance with \eqref{Hamiltonian}, the Hamiltonian of $(FP_{2a})$ is given by 
		\begin{equation}\label{Hamiltonian_FP_2a}
		\mathcal{H}(t, x, p, u)=-aup_1-e^{-\lambda t}(x_1+u)p_2 \quad \forall(t, x, p, u)\in [t_0, T]\times \R^2\times \R^2 \times \R.
		\end{equation}
		By \eqref{h-partial subdiff}, the partial hybrid subdifferential of $h$ at $(t, x)\in [t_0, T]\times \R^2$ is given by
		\begin{equation}\label{hybrid subdiff_FP_2a}
		\partial^>_x h(t, x)=
		\begin{cases}
		\emptyset, &\quad \mbox{if} \ x_1<1\\
		\{(1,0)\}, & \quad \mbox{if} \ x_1\geq 1.
		\end{cases}
		\end{equation}
		
	\textit{From now on, let $(\bar x, \bar u)$ be a $W^{1,1}$ local minimizer for $(FP_{2a})$}. 
	
	Since the  assumptions (H1)--(H4) of Theorem~\ref{V_thm9.3.1 necessary condition} are satisfied for $(FP_{2a})$, by that theorem one can find $p\in W^{1,1}([t_0, T]; \R^2)$, $\gamma \geq 0$, $\mu \in C^\oplus(t_0, T)$, and a Borel measurable  function $\nu:[t_0, T]\to  \R^2$ such that $(p, \mu, \gamma)\neq (0, 0, 0)$, and for $q(t):=p(t)+\eta(t)$ with \begin{equation}\label{1st_formula_for_eta}\eta(t):=
		\displaystyle\int_{[t_0, t)}\nu(\tau)d\mu(\tau)\quad\ (\forall t\in [t_0, T))\end{equation} and
		 \begin{equation}\label{2nd_formula_for_eta}
		 \eta(T):=\displaystyle\int_{[t_0, T]}\nu(\tau)d\mu(\tau),
		 \end{equation}
		 conditions (i)--(iv) in Theorem~\ref{V_thm9.3.1 necessary condition} hold true.
		
		\textbf{Condition (i)}: Note that
		\begin{align*}
		&  \mu \{t\in [t_0, T] \,:\, \nu(t) \notin \partial^>_x h(t, \bar x(t))\}\\ &= \mu \{t\in [t_0, T] \,:\, \partial^>_x h(t, \bar x(t))=\emptyset\}+\mu \{t\in [t_0, T] \,:\, \partial^>_x h(t, \bar x(t))\neq\emptyset,\; \nu(t) \notin \partial^>_x h(t, \bar x(t))\}.
		\end{align*}
		Since $\bar x_1(t)\leq 1$ for every $t$, combining this with \eqref{hybrid subdiff_FP_2a} gives
		\begin{align*}
		& \mu \{t\in [t_0, T] \,:\, \nu(t) \notin \partial^>_x h(t, \bar x(t))\}\\ &= \mu \{t\in [t_0, T] \,:\,\bar x_1(t)<1\} + \mu \{t\in [t_0, T] \,:\, \bar x_1(t)=1,\; \nu(t) \neq (1, 0)\}.
		\end{align*}
	So, from (i)  it follows that 
		\begin{equation}\label{1st_condition_for_mu}\mu \{t\in [t_0, T] \,:\,\bar x_1(t)<1\}=0\end{equation} and \begin{equation}\label{2nd_condition_for_mu}\mu \big\{t\in [t_0, T] \,:\, \bar x_1(t)=1,\; \nu(t) \neq (1, 0)\big\}=0.\end{equation}
		
		\textbf{Condition (ii)}: By \eqref{Hamiltonian_FP_2a}, $\mathcal{H}$ is differentiable in $x$ and $\partial_x\mathcal{H}(t, x, p, u)=\{(-e^{-\lambda t}p_2, 0)\}$ for all $(t, x, p, u)\in [t_0, T]\times \R^2\times \R^2 \times \R$. Thus, (ii) implies that $-\dot p(t) =(-e^{-\lambda t}q_2(t), 0)$ for a.e. $t\in [t_0, T]$. Hence, $\dot p_1(t) = e^{-\lambda t}q_2(t)$  for a.e. $t\in [t_0, T]$ and $p_2(t)$ is  a constant for all $t\in [t_0, T]$.
		
		\textbf{Condition (iii)}: By the formulas for $g$ and $C$, $\partial g(\bar x(t_0), \bar x(T))=\{(0, 0, 0, 1)\}$ and $N_C(\bar x(t_0), \bar x(T))=\R^2\times\{(0, 0)\}$. Thus, (iii) yields
		$$(p(t_0), -q(T))\in \{(0, 0, 0, \gamma)\}+\R^2\times\{(0, 0)\},$$ which means that
		 $q_1(T)=0$ and $q_2(T)=-\gamma$.
		
		\textbf{Condition (iv)}: By \eqref{Hamiltonian_FP_2a}, from (iv) one gets
		\begin{equation*}
		-a\bar u(t)q_1(t)-e^{-\lambda t}[\bar x_1(t)+\bar u(t)]q_2(t)=\max_{u\in [-1, 1]}\left\{-auq_1(t)-e^{-\lambda t}[\bar x_1(t)+u]q_2(t) \right\}\ \mbox{a.e.}\; t\in [t_0,T]
		\end{equation*}
		or, equivalently,
		\begin{equation}\label{min_condition_FP_2a}
		[aq_1(t)+e^{-\lambda t}q_2(t)]\bar u(t)=\min_{u\in [-1, 1]}\left\{[aq_1(t)+e^{-\lambda t}q_2(t)]u\right\}\ \mbox{a.e.}\; t\in [t_0,T].
		\end{equation}
		
		Thanks to Proposition~\ref{lemma_basic_property_1} and the above analysis of Conditions (i)--(iv), we will be able to prove next statement.
		
			\begin{Proposition}\label{lemma_interior_traj}
				Suppose that $[\tau_1, \tau_2]$ is a subsegment of $[t_0, T]$ with $h (t, \bar x(t))<0$ for all $t\in [\tau_1, \tau_2]$. Then, the curve $t\mapsto \bar x_1(t)$, $t\in [\tau_1, \tau_2]$, cannot have more than one turning point. To be more precise, the curve must be of one of the following three categories {\rm C1$-$C3}:
				\begin{equation}\label{lemma_interior_traj_1}
					\bar x_1(t)=\bar x_1(\tau_1)+a(t-\tau_1), \quad t\in [\tau_1, \tau_2],
				\end{equation}
				\begin{equation}\label{lemma_interior_traj_2}
					\bar x_1(t)=\bar x_1(\tau_1)-a(t-\tau_1), \quad t\in [\tau_1, \tau_2],
				\end{equation}
				and
				\begin{equation} \label{lemma_interior_traj_3}
					\bar x_1(t)=
					\begin{cases}
						\bar x_1(\tau_1)+a(t-\tau_1), \quad & t\in [\tau_1, t_{\zeta}]\\
						\bar x_1(t_{\zeta})-a(t-t_\zeta), & t\in (t_{\zeta}, \tau_2],
					\end{cases}
				\end{equation} where  $t_{\zeta}$ is a certain point in $(\tau_1, \tau_2)$ (see {\rm Fig.~1--3}).		
				\begin{figure}[!ht]
					\begin{minipage}[b]{0.5\textwidth}
						\centering
						\includegraphics[scale=.15]{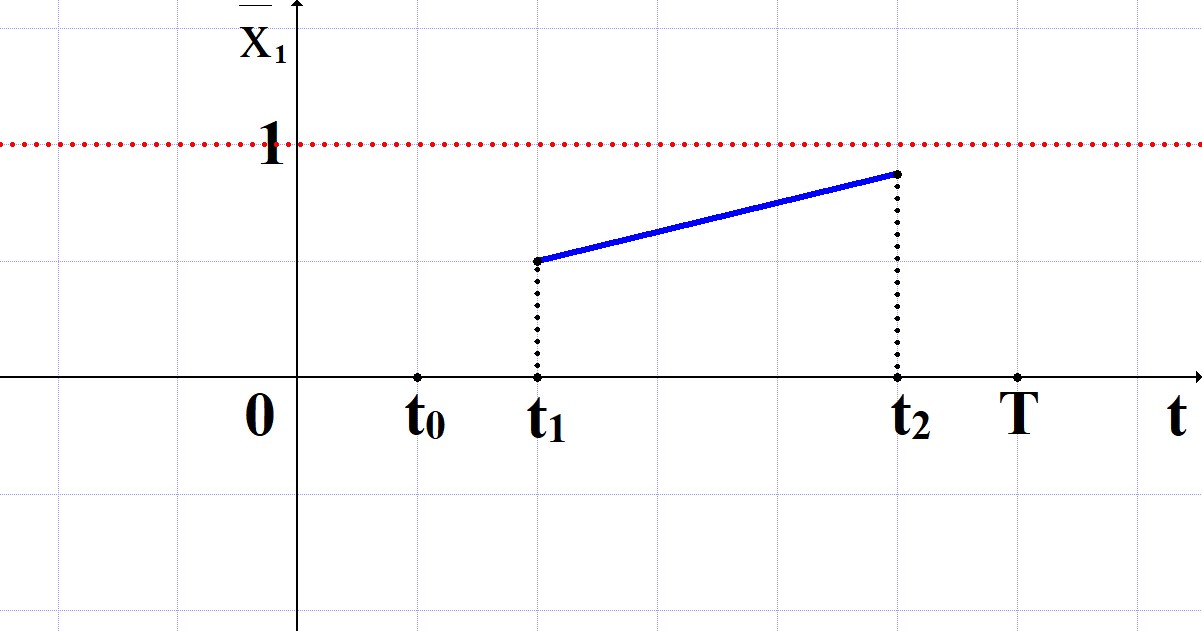}
						\caption{Category C1} %\label{fig1}
					\end{minipage}
					\hfill
					\begin{minipage}[b]{0.5\textwidth}
						\centering
						\includegraphics[scale=.15]{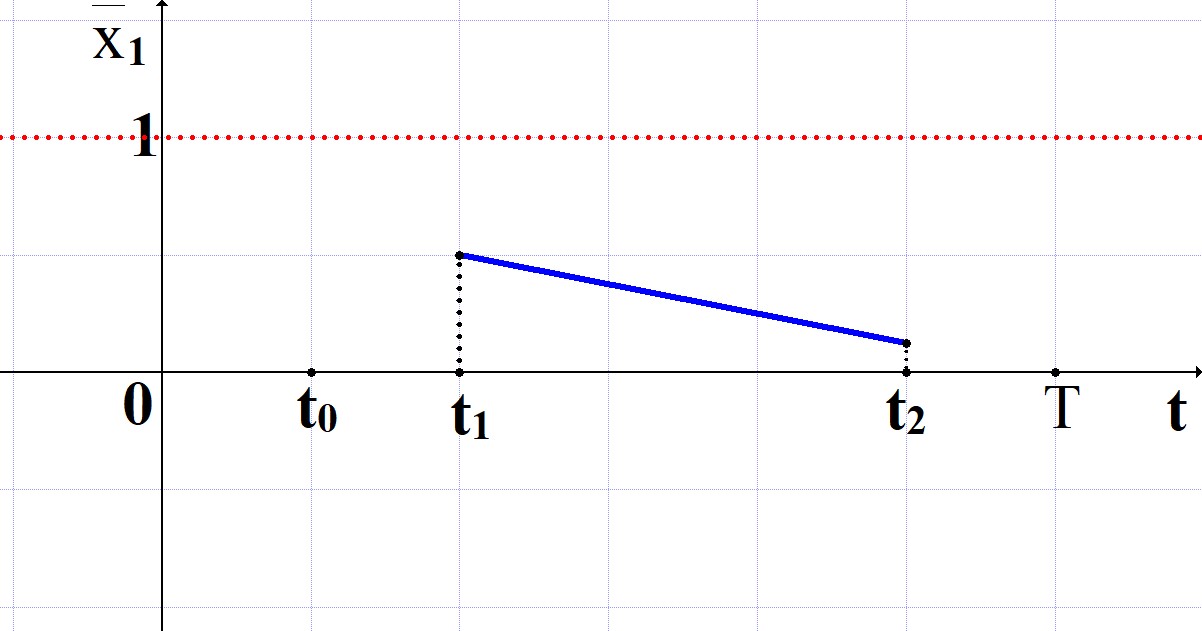}
						\caption{Category C2} %\label{fig2}
					\end{minipage}
				\end{figure}
				\begin{figure}[!ht]
					\centering
					\includegraphics[scale=.15]{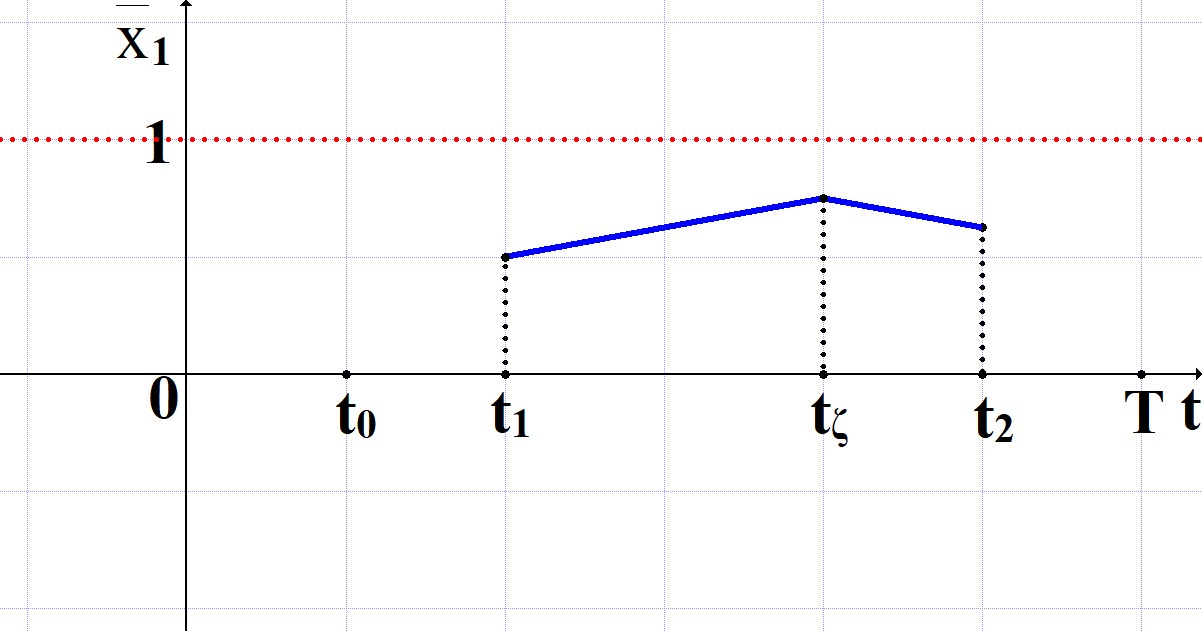}
					\caption{Category C3} %\label{fig3}
				\end{figure}
			\end{Proposition}	
			\begin{proof}
				Suppose that $[\tau_1, \tau_2]$ is a subsegment of $[t_0, T]$ with $h (t, \bar x(t))<0$ for all $t\in [\tau_1, \tau_2]$, i.e., $\bar x_1(t)< 1$ for all $t\in [\tau_1, \tau_2]$. Then, it follows from Proposition~\ref{lemma_basic_property_1} that the restriction of $(\bar x, \bar u)$ on $[\tau_1, \tau_2]$ is a $W^{1,1}$ local minimizer for $(FP_{2a})|_{[\tau_1, \tau_2]}$. Since the latter satisfies the assumptions (H1)--(H4) of Theorem~\ref{V_thm9.3.1 necessary condition}, by that theorem one finds $\widetilde p\in W^{1,1}([\tau_1, \tau_2]; \R^2)$, $\widetilde \gamma \geq 0$, $\widetilde \mu \in C^\oplus(\tau_1, \tau_2)$, and a Borel measurable  function $\widetilde \nu:[\tau_1, \tau_2]\to  \R^2$ with the property $(\widetilde p, \widetilde \mu, \widetilde \gamma)\neq (0, 0, 0)$, and for $\widetilde q(t):=\widetilde p(t)+\widetilde \eta(t)$ with \begin{equation}\label{1st_formula_for_w_eta_FP3b}\widetilde \eta(t):=
					\displaystyle\int_{[\tau_1, \tau_2)}\widetilde \nu(\tau)d\widetilde \mu(\tau)\quad\ (\forall t\in [\tau_1, \tau_2))\end{equation} and
				\begin{equation}\label{2nd_formula_for_w_eta_FP3b}\widetilde \eta(\tau_2):=\displaystyle\int_{[\tau_1, \tau_2]}\widetilde \nu(\tau)d\widetilde \mu(\tau),\end{equation} the conditions (i)--(iv) in Theorem~\ref{V_thm9.3.1 necessary condition} hold true, provided that $t_0, T, p, \mu, \gamma, \nu, \eta$, and $q$ are changed respectively to $\tau_1, \tau_2, \widetilde p, \widetilde \mu, \widetilde \gamma, \widetilde \nu, \widetilde \eta$, and $\widetilde q$.
				
				By Condition (i), one has
				\begin{equation}\label{1st_condition_for_w_mu_FP3b}\begin{cases}\widetilde \mu \{t\in [\tau_1, \tau_2] \,:\,\bar x_1(t)<1\}=0,\\
						\widetilde \mu \big\{t\in [\tau_1, \tau_2] \,:\, \bar x_1(t)=1,\; \widetilde \nu(t) \neq (1, 0)\big\}=0.\end{cases}\end{equation}
				
				By Condition (ii), $\dot{\widetilde p_1}(t) = e^{-\lambda t}\widetilde q_2(t)$  for a.e. $t\in [\tau_1, \tau_2]$ and $\widetilde p_2(t)$ is  a constant for all $t\in [\tau_1, \tau_2]$.	
				
				Since $N_{\widetilde C}(\bar x(\tau_1), \bar x(\tau_2))=\R^3\times\{0\}$, by Condition (iii) one has
				$$(\widetilde p(\tau_1), -\widetilde q(\tau_2))\in \{(0, 0, 0, \widetilde\gamma)\}+\R^3\times\{0\}.$$
				This amounts to saying that $\widetilde q_2(\tau_2)=-\widetilde \gamma$.	
				
				Condition (iv) means that
				\begin{equation}\label{min_condition_FP3c}
					[a\widetilde q_1(t)+e^{-\lambda t}\widetilde q_2(t)]\bar u(t)=\min_{u\in [-1, 1]}\left\{[a\widetilde q_1(t)+e^{-\lambda t}\widetilde q_2(t)]u\right\}\ \, \mbox{a.e.}\ t\in [\tau_1, \tau_2].
				\end{equation}
				
				Since $\bar x_1(t)<1$ for all $t\in [\tau_1, \tau_2]$, \eqref{1st_condition_for_w_mu_FP3b} yields $\widetilde \mu([\tau_1, \tau_2])=0$, i.e., $\widetilde \mu=0$. Combining this with  \eqref{1st_formula_for_w_eta_FP3b} and \eqref{2nd_formula_for_w_eta_FP3b}, one gets $\widetilde \eta(t)=0$ for all $t\in [\tau_1, \tau_2]$. Thus, the relation $\widetilde q(t)=\widetilde p(t)+\widetilde \eta(t)$ implies that $\widetilde q(t)=\widetilde p(t)$ for every $t\in [\tau_1, \tau_2]$. Therefore, together with the Lebesgue Theorem \cite[Theorem~6, p.~340]{Kolmogorov_Fomin_1970}, the properties of $\widetilde p(t)$ and $\widetilde q(t)$ established in the above analyses of the conditions~(ii) and~(iii) give $\widetilde p_2(t)=\widetilde q_2(t)=-\widetilde \gamma$ and $\widetilde p_1(t)=\widetilde q_1(t)=\dfrac{\widetilde\gamma}{\lambda}e^{-\lambda t}+\zeta$ for all $t\in [\tau_1, \tau_2]$, where $\zeta$ is a constant. 
				Substituting these formulas for $\widetilde q_1(t)$ and $\widetilde q_2(t)$ to ~\eqref{min_condition_FP3c}, we have
				\begin{equation*}
					\Big[a(\dfrac{\widetilde\gamma}{\lambda}e^{-\lambda t}+\zeta)-\widetilde\gamma e^{-\lambda t}\Big]\bar u(t)=\min_{u\in [-1, 1]}\left\{[a(\dfrac{\widetilde\gamma}{\lambda}e^{-\lambda t}+\zeta)-\widetilde\gamma e^{-\lambda t}]u\right\} \quad \mbox{a.e. } t\in [\tau_1, \tau_2]
				\end{equation*}
				or, equivalently, 
				\begin{equation}\label{min_condition_FP3d}
					[\widetilde \gamma (\dfrac{a}{\lambda}-1)e^{-\lambda t}+a\zeta]\bar u(t)=\min_{u\in [-1, 1]}\left\{[\widetilde \gamma (\dfrac{a}{\lambda}-1)e^{-\lambda t}+a\zeta]u\right\} \quad \mbox{a.e. } t\in [\tau_1, \tau_2].
				\end{equation}
				
				Set $\widetilde\varphi(t)=\widetilde \gamma (\dfrac{a}{\lambda}-1)e^{-\lambda t}+a\zeta$ for all $t\in [\tau_1, \tau_2]$. 
				
				If $\widetilde\gamma=0$, then $\widetilde\varphi(t)\equiv a\zeta$ on $[\tau_1, \tau_2]$. Since $a>0$, the condition $(\widetilde p, \widetilde \mu, \widetilde \gamma)\neq (0, 0, 0)$ implies that $\zeta\neq 0$. If $\zeta>0$, then  $\widetilde\varphi(t)>0$ for all $t\in [\tau_1, \tau_2]$. If $\zeta<0$, then $\widetilde\varphi(t)<0$ for all $t\in [\tau_1, \tau_2]$. Thus, if $\zeta>0$, then \eqref{min_condition_FP3d} implies that $\bar u(t)=-1$ a.e. $t\in [\tau_1, \tau_2]$.  Similarly, if  $\zeta<0$, then $\bar u(t)=1$ a.e. $t\in [\tau_1, \tau_2]$. Hence, applying the Lebesgue Theorem \cite[Theorem~6, p.~340]{Kolmogorov_Fomin_1970} to the absolutely continuous function $\bar x_1(t)$, one has 
				\begin{equation} \label{interior_traj_4}
					\bar x_1(t)=\bar x_1(\tau_1)+a(t-\tau_1) \quad (\forall t\in [\tau_1, \tau_2])
				\end{equation}
				in the first case, and
				\begin{equation} \label{interior_traj_5}
					\bar x_1(t)=\bar x_1(\tau_1)-a(t-\tau_1) \quad (\forall t\in [\tau_1, \tau_2])
				\end{equation} 	in the second case.			
				
				If $\widetilde\gamma>0$ then, due to the assumption $a>\lambda>0$, $\widetilde\varphi$ is strictly decreasing on $[\tau_1, \tau_2]$. When there exists $t_{\zeta}\in (\tau_1, \tau_2)$ such that $\widetilde\varphi(t_{\zeta})=0$, one has $\widetilde\varphi(t)>0$ for $t\in (\tau_1, t_{\zeta})$ and $\widetilde\varphi(t)<0$ for $t\in (t_{\zeta}, \tau_2)$. Hence, \eqref{min_condition_FP3d} forces $\bar u(t)=-1$ a.e. $t\in [\tau_1, t_{\zeta}]$ and $\bar u(t)=1$ a.e. $t\in [t_{\zeta}, \tau_2]$.  Thus, by the cited above Lebesgue Theorem,
				\begin{equation*} \label{interior_traj_6}
					\bar x_1(t)=
					\begin{cases}
						\bar x_1(\tau_1)+a(t-\tau_1), \quad&\mbox{for}\; t\in [\tau_1, t_{\zeta}]\\
						\bar x_1(t_{\zeta})-a(t-t_\zeta), &\mbox{for}\; t\in (t_{\zeta}, \tau_2].
					\end{cases}
				\end{equation*} As $\bar x_1(t)<1$ for all $t\in [\tau_1, \tau_2]$, one must have $\bar x_1(t_{\zeta})<1$, i.e., $t_{\zeta}<\tau_1+a^{-1}(1-\bar x_1(\tau_1))$.  When $\widetilde\varphi(t)>0$ for all $t\in (\tau_1, \tau_2)$, condition \eqref{min_condition_FP3d} implies that $\bar u(t)=-1$ a.e. $t\in [\tau_1, \tau_2]$. So,  $\bar x_1(t)$ is defined by \eqref{interior_traj_4}. When $\widetilde\varphi(t)<0$ for all $t\in (\tau_1, \tau_2)$,  condition \eqref{min_condition_FP3d} implies that $\bar u(t)=1$ a.e. $t\in [\tau_1, \tau_2]$. Hence,  $\bar x_1(t)$ is defined by \eqref{interior_traj_5}.

				In summary, for any $\tau_1, \tau_2$ with $t_0\leq \tau_1<\tau_2\leq T$ and $\bar x_1(t)<1$ for all $t\in [\tau_1, \tau_2]$, the curve $t\mapsto \bar x_1(t)$, $t\in [\tau_1, \tau_2]$, cannot have more than one turning point. Namely, the curve must be of one of the three categories \eqref{lemma_interior_traj_1}--\eqref{lemma_interior_traj_3}.
			\end{proof}	
					
		To proceed furthermore, put $ {\mathcal T}_1:=\{t\in [t_0,T]\; :\; \bar x_1(t)=1\}.$ Since $\bar x_1(t)$ is a continuous function, ${\mathcal T}_1$ is a compact set (which may be empty).
				
		{\bf Case 1:} \textit{${\mathcal T}_1=\emptyset$, i.e., $\bar x_1 (t) <1$ for all $t\in [t_0, T]$.} Then, by \eqref{1st_condition_for_mu} one has $\mu([t_0, T])=0$, i.e., $\mu=0$. Combining this with  \eqref{1st_formula_for_eta} and \eqref{2nd_formula_for_eta}, one gets $\eta(t)=0$ for all $t\in [t_0, T]$. Thus, the relation $q(t)=p(t)+\eta(t)$ allows us to have $q(t)=p(t)$ for every $t\in [t_0, T]$. Therefore, together with the Lebesgue Theorem \cite[Theorem~6, p.~340]{Kolmogorov_Fomin_1970}, the properties of $p(t)$ and $q(t)$ established in the above analyses of the conditions~(ii) and~(iii) give $$p_2(t)=q_2(T)=-\gamma\quad\ (\forall t\in [t_0, T])$$ and $$p_1(t)=p_1(T)+\int_T^t  \dot p_1(\tau)d\tau=q_1(T)+\int_T^t  \big(-\gamma e^{-\lambda \tau}\big)d\tau= \dfrac{\gamma}{\lambda}\big(e^{-\lambda t}-e^{-\lambda T}\big)$$ for all $t\in [t_0, T]$. Now, observe that substituting $q(t)=p(t)$ into \eqref{min_condition_FP_2a} yields
		\begin{equation}\label{min_condition_FP2b}
		[ap_1(t)+e^{-\lambda t}p_2(t)]\bar u(t)=\min_{u\in [-1, 1]}\left\{[ap_1(t)+e^{-\lambda t}p_2(t)]u\right\} \quad \mbox{a.e. } t\in [t_0, T].
		\end{equation}
		Setting $\varphi(t)=ap_1(t)+e^{-\lambda t}p_2(t)$ for $t\in [t_0, T]$ and using the above formulas of $p_1(t)$ and $p_2(t)$, we have 
 		\begin{align*}
		\varphi(t) =a\dfrac{\gamma}{\lambda}\big(e^{-\lambda t}-e^{-\lambda T}\big)-\gamma e^{-\lambda t}=\gamma (\dfrac{a}{\lambda}-1)e^{-\lambda t}-\gamma\dfrac{a}{\lambda} e^{-\lambda T}
		\end{align*}
		for $t\in [t_0, T]$. Due to the condition $(p, \gamma, \mu)\neq 0$, one must have $\gamma>0$. Moreover, the assumption $a>\lambda>0$ implies $\dfrac{a}{\lambda}>1$. Thus, the function $\varphi(t)$  is decreasing on  $[t_0, T]$. In addition, it is clear that $\varphi(T)=-\gamma e^{-\lambda T}<0$, and $\varphi(t)=0$ if and only if $t=\bar t$, where
		 \begin{equation}\label{special_time}\bar t:=T-\dfrac{1}{\lambda}\ln \dfrac{a}{a-\lambda}.
		 \end{equation}The assumption $a>\lambda>0$ implies that $\bar t<T$. Note that the number $\rho:=\dfrac{1}{\lambda}\ln \dfrac{a}{a-\lambda}$ does not depend on the initial time $t_0$ and the terminal time $T$.
		 
	If $t_0\geq \bar t$, then one has $\varphi(t)<0$ for all $t\in (t_0, T)$. This situation happens if and only if $T-t_0\leq\rho$ (the time interval of the optimal control problem is rather small). Clearly, condition \eqref{min_condition_FP2b} forces $\bar u(t)=1$ a.e. $t\in [t_0, T]$. Since \eqref{state control system_FP_2a} is fulfilled for $x(t)=\bar x(t)$ and $u(t)=\bar u(t)$, applying the Lebesgue Theorem \cite[Theorem~6, p.~340]{Kolmogorov_Fomin_1970} to the absolutely continuous function $\bar x_1(t)$, one has
		\begin{equation}\label{descending_trajectory1}\bar x_1(t)=\bar x_1(t_0)+\int_{t_0}^t\dot{\bar x}_1(\tau)d\tau=\bar x_1(t_0)+\int_{t_0}^t(-a\bar u(\tau))d\tau=x_0-a(t-t_0)\end{equation} for all $t\in [t_0, T]$. In addition, by \eqref{phase_second_component} one finds that \begin{equation}\label{descending_trajectory2}\bar x_2(t)=\int_{t_0}^{t} \big[-e^{-\lambda\tau}(\bar x_1(\tau)+\bar u(\tau))\big] d\tau=\int_{t_0}^{t} \big[-e^{-\lambda \tau}\big(x_0-a(\tau-t_0)+1\big)\big] d\tau\end{equation} for all $t\in [t_0, T]$.
			
		{If $t_0<\bar t$, then $\varphi(t)>0$ for $t\in (t_0, \bar t)$ and $\varphi(t)<0$ for $t\in (\bar t, T)$.  This situation happens if and only if $T-t_0>\rho$ (the time interval of the optimal control  problem is large enough).  Condition \eqref{min_condition_FP2b} yields $\bar u(t)=-1$ for a.e. $t\in [t_0, \bar t]$ and $\bar u(t)=1$ for a.e. $t\in [\bar t, T]$. Hence, by the above-cited Lebesgue Theorem, one has 
		\begin{equation}\label{up_then_down_traj}
		\bar x_1(t)=
		\begin{cases}
		x_0+a(t-t_0),\quad &\mbox{if }\,  t\in [t_0, \bar t]\\
		\bar x_1(\bar t)-a(t-\bar t),\quad &\mbox{if }\,  t\in (\bar t, T].
		\end{cases}
		\end{equation}
		Therefore, from \eqref{phase_second_component}, we have
		\begin{align}\label{up_then_down_traj2}
	\bar x_2(t)=\begin{cases}
	\displaystyle\int_{t_0}^{t} \big[-e^{-\lambda \tau}\big(x_0+a(\tau-t_0)+1\big)\big] d\tau, \quad &\mbox{if }\, t\in [t_0, \bar t]\\
	\displaystyle\int_{t_0}^{t} \big[-e^{-\lambda \tau}\big(\bar x_1(\bar t)-a(\tau-\bar t)+1]d\tau,  &\mbox{if }\, t\in (\bar t, T].
	\end{cases}
	\end{align}
		Noting that $\bar x (t) <1$ for all $t\in [t_0, T]$ by our assumption, we must have  $\bar x_1(\bar t)<1$, i.e., $\bar t<t_0+a^{-1}(1-x_0)$. Since $\bar t=T- \rho$, the last inequality is equivalent to $T-t_0<\rho+a^{-1}(1-x_0)$.}
		
			{Thus, if  ${\mathcal T}_1=\emptyset$ and $T-t_0\leq\rho$, then the unique process $(\bar x,\bar u)$ suspected for a $W^{1,1}$ local optimizer of $(FP_{2a})$ is the one with $\bar u(t)=1$ a.e. $t\in [t_0, T]$, $\bar x(t)=(\bar x_1(t),\bar x_2(t))$, where $\bar x_1(t)$ and  $\bar x_2(t)$ are given respectively by \eqref{descending_trajectory1} and \eqref{descending_trajectory2}. Otherwise, if ${\mathcal T}_1=\emptyset$ and $$\rho<T-t_0<\rho+a^{-1}(1-x_0),$$ then the unique process $(\bar x,\bar u)$ serving as a $W^{1,1}$ local optimizer of $(FP_{2a})$ is the one with $\bar u(t)=-1$ for a.e. $t\in [t_0, \bar t]$ and $\bar u(t)=1$ for a.e. $t\in [\bar t, T]$, $\bar x(t)=(\bar x_1(t),\bar x_2(t))$, where $\bar x_1(t)$ and  $\bar x_2(t)$ are defined respectively by \eqref{up_then_down_traj} and \eqref{up_then_down_traj2}. The situation where ${\mathcal T}_1=\emptyset$ and $T-t_0\geq \rho+a^{-1}(1-x_0)$ cannot occur.  The situation where ${\mathcal T}_1=\emptyset$ and $x_0\geq 1-a(\bar t-t_0)$ also cannot occur.}
	
			Now, suppose that ${\mathcal T}_1\neq\emptyset$, i.e., there exists $t\in [t_0, T]$ with the property $\bar x (t) =1$. Setting
				\begin{equation*}\label{two_alphas} \alpha_1:=\min \{t\in [t_0,T]\; :\; \bar x_1(t)=1\},\quad \alpha_2:=\max\{t\in [t_0,T]\; :\; \bar x_1(t)=1\},
				\end{equation*} we have $t_0\leq\alpha_1\leq\alpha_2\leq T$. The following situations can occur.
				
				{\bf Case 2:} \textit{$t_0<\alpha_1=\alpha_2=T$, i.e., $\bar x_1(t)<1$ for $t\in [t_0,T)$ and $\bar x_1(T)=1$}. Clearly,~\eqref{1st_condition_for_mu} means that $\mu([t_0, T))=0$. Moreover, if $\nu(T)\neq (1, 0)$, then from \eqref{2nd_condition_for_mu} it follows that $\mu(\{T\})=0$. So, we have $\mu([t_0, T])=\mu([t_0, T))+\mu(\{T\})=0$, i.e., $\mu=0$. Hence, we can repeat the arguments already used in Case 1 to prove that either $\bar x_1(t)=x_0-a(t-t_0)$ for all $t\in [t_0, T]$, or \begin{equation*}
				\bar x_1(t)=
				\begin{cases}
				x_0+a(t-t_0),\quad &\mbox{if }\,  t\in [t_0, \bar t]\\
				\bar x_1(\bar t)-a(t-\bar t),\quad &\mbox{if }\,  t\in (\bar t, T].
				\end{cases}
				\end{equation*} In particular, either we have $\bar x_1(T)=x_0-a(T-t_0)<1$, or $\bar x_1(T)=\bar x_1(\bar t)-a(T-\bar t)<1$. Both instances are impossible, because $\bar x_1(T)=1$. So, the situation $\nu(T)\neq (1, 0)$ is excluded; thus $\nu(T)=(1, 0)$. 
				
				From \eqref{1st_formula_for_eta} and \eqref{2nd_formula_for_eta}, one gets $\eta(t)=0$ for $t\in [t_0, T)$ and $\eta(T)=(\mu(T)-\mu(T-0), 0),$ where $\mu(T-0)$ denotes the left limit of $\mu$ at $T$. Therefore, the relation $q(t)=p(t)+\eta(t)$, which holds for every $t\in [t_0, T]$, yields $q_1(t)=p_1(t)$ for $t\in [t_0, T)$, $q_1(T)=p_1(T)+\mu(T)-\mu(T-0)$, and $q_2(t)=p_2(t)$ for $t\in [t_0, T]$. Combining this with the above results of our analyses of the conditions (ii) and (iii), we have $p_2(t)=-\gamma$ and $p_1(t)=\dfrac{\gamma}{\lambda}e^{-\lambda t}+\zeta$ for all $t\in [t_0, T]$, with~$\zeta$ being a constant. Since $q(t)$ equals to $p(t)$ everywhere on $[t_0, T]$, except possibly for $t=T$, condition~\eqref{min_condition_FP_2a} implies that
						\begin{equation}\label{min_condition_FP2c}
						[ap_1(t)+e^{-\lambda t}p_2(t)]\bar u(t)=\min_{u\in [-1, 1]}\left\{[ap_1(t)+e^{-\lambda t}p_2(t)]u\right\} \quad \mbox{a.e. } t\in [t_0, T].
						\end{equation}
				As in Case 1, we set $\varphi(t)=ap_1(t)+e^{-\lambda t}p_2(t)$ for every $t\in [t_0, T]$. Here one has 
						\begin{align*}
						\varphi(t) =a\big(\dfrac{\gamma}{\lambda}e^{-\lambda t}+\zeta\big)-\gamma e^{-\lambda t}=\gamma (\dfrac{a}{\lambda}-1)e^{-\lambda t}+a\zeta
						\end{align*}
						for all $t\in [t_0, T]$.
				Since $\dfrac{a}{\lambda}>1$, the function $\varphi(t)$  is decreasing on  $[t_0, T]$. Besides, since $\mu(T)-\mu(T-0)\geq 0,$ $q_1(T)=p_1(T)+\mu(T)-\mu(T-0)$, and $q_1(T)=0$, we have $p_1(T)\leq 0$. So, $\varphi(T)=ap_1(T)-\gamma e^{-\lambda T}<0$. If $\varphi(t)<0$ for all $t\in (t_0,T)$, then by~\eqref{min_condition_FP2c} one has $\bar u(t)=1$ for a.e.  $t\in [t_0,T]$. 
				So, as it has been done in \eqref{descending_trajectory1}, we have $\bar x_1(t)=x_0-a(t-t_0)$ for all $t\in [t_0,T]$. This yields $\bar x_1(T)<x_0<1$. We have arrived at a contraction. Now, suppose that there exists $\bar t_\zeta\in [t_0,T)$ satisfying $\varphi(\bar t_\zeta)=0$.  Then $\varphi(t)>0$ for $t\in (t_0, \bar t_\zeta)$ and $\varphi(t)<0$ for $t\in (\bar t_\zeta, T)$. Thus, \eqref{min_condition_FP2b} yields $\bar u(t)=-1$ for a.e. $t\in [t_0, \bar t_\zeta]$ and $\bar u(t)=1$ for a.e. $t\in [\bar t_\zeta, T]$. Hence, applying the Lebesgue Theorem \cite[Theorem~6, p.~340]{Kolmogorov_Fomin_1970} to the absolutely continuous function $\bar x_1(t)$, one has $\bar x_1(t)=a(t-t_0)+x_0$ for all $t\in [t_0, \bar t_\zeta]$ and $\bar x_1(t)=-a(t-\bar t_\zeta)+\bar x_1(\bar t_\zeta)$ for every $t\in [\bar t_\zeta, T]$. As $\bar x (t) <1$ for all $t\in [t_0, T]$ by our assumption, we must have  $\bar x_1(\bar t_\zeta)<1$.  Then we get $\bar x_1(T)=-a(T-\bar t_\zeta)+\bar x_1(\bar t_\zeta)<1,$ which is impossible.
				
				{\bf Case 3:} \textit{$t_0=\alpha_1=\alpha_2<T$, i.e., $x_0=1$ and $\bar x_1(t)<1$ for $t\in (t_0,T]$}. Let $\bar\varepsilon>0$ be such that $t_0+\bar\varepsilon< T$. For any $k\in\N$ with $k^{-1}\in (0,\bar\varepsilon)$, by Proposition~\ref{lemma_basic_property_2} we know that the restriction of $(\bar x, \bar u)$ on $[t_0+k^{-1}, T]$ is a $W^{1,1}$ local minimizer for the Mayer problem~$(FP_{2b})$, which is obtained from $(FP_{2a})$ by replacing $t_0$ with $t_0+k^{-1}$. Since $\bar x_1(t)<1$ for all $t\in [t_0+k^{-1},T]$, we can  repeat the arguments already used in Case 1 to get that  either $\bar x_1(t)=\bar x_1(t_0+k^{-1})-a(t-t_0-k^{-1})$ for all $t\in [t_0+k^{-1}, T]$, or 
				\begin{equation*}
			    \bar x_1(t)=
				\begin{cases}
				\bar x_1(\bar t)+a(t-\bar t),\quad &\mbox{if }\,  t\in [t_0+k^{-1}, \bar t]\\
				\bar x_1(\bar t)-a(t-\bar t),\quad &\mbox{if }\,  t\in (\bar t, T]
				\end{cases}
				\end{equation*} with $\bar t=T-\rho$, $\bar t\in  [t_0+k^{-1}, T]$, and $\bar x_1(\bar t)<1$. By the Dirichlet principle, there must exist an infinite number of indexes $k$ with $k^{-1}\in (0,\bar\varepsilon)$ such that $\bar x_1(t)$ has the first form (resp., the second form). Without loss of generality, we may assume that this happens for all $k$ with $k^{-1}\in (0,\bar\varepsilon)$. If the first situation occurs, then by letting $k\to\infty$ we can assert that $\bar x_1(t)=1-a(t-t_0)$ for all $t\in [t_0, T]$. If the second situation occurs, then we have
				\begin{equation*}
				x_0=\displaystyle\lim_{k\to\infty}\bar x_1(t_0+k^{-1})=
						\displaystyle\lim_{k\to\infty}\big[\bar x_1(\bar t)+a(t_0+k^{-1}-\bar t)\big]=\bar x_1(\bar t)+a(t_0-\bar t).
				\end{equation*} Since $\bar x_1(\bar t)+a(t_0-\bar t)\leq \bar x_1(\bar t)<1$ and $x_0=1$, we have arrived at a contradiction.
								
				{\bf Case 4:} \textit{$t_0<\alpha_1\leq\alpha_2<T$}. Then, $\bar x_1(\alpha_1)=\bar x_1(\alpha_2)=1$, $\bar x_1(t)<1$ for $t\in [t_0,\alpha_1)\cup  (\alpha_2,T]$. To find a formula for $(\bar x, \bar u)$ on $[\alpha_2, T]$, observe from Proposition~\ref{lemma_basic_property_2} that the restriction of $(\bar x, \bar u)$ on $[\alpha_2, T]$ is a $W^{1,1}$ local minimizer for the Mayer problem obtained from $(FP_{2a})$ by replacing $t_0$ with $\alpha_2$. Thus, the result in Case 3 applied to the process $(\bar x(t), \bar u(t))$, $t\in [\alpha_2, T]$, implies that $\bar x_1(t)=1-a(t-\alpha_2)$ and $\bar x_2(t)=\displaystyle\int_{\alpha_2}^{t} \big[-e^{-\lambda \tau}\big(1-a(\tau-\alpha_2)+1\big)\big] d\tau$ for all $t\in [\alpha_2, T]$. To obtain a formula for $(\bar x, \bar u)$ on $[t_0, \alpha_2]$,  consider the following two subcases.
		
			\underline{\textit{Subcase 4a}}: $t_0<\alpha_1=\alpha_2<T$. Here we have $\bar x_1(\alpha_1)=1$ and $\bar x_1(t)<1$ for all $t\in [t_0, T]\setminus\{\alpha_1\}$. To find a formula for $\bar x_1(.)$ on $[t_0, \alpha_1]$, we temporarily fix a value $\alpha \in (t_0, \alpha_1)$ (later, we will let  $\alpha$ converge to $\alpha_1$). Since $\bar x_1(t)<1$ for all $[t_0, \alpha]$, applying Proposition~\ref{lemma_interior_traj} with $\tau_1:=t_0$ and $\tau_2:=\alpha$, we can assert that the restriction of $\bar x_1(.)$ on $[t_0, \alpha]$ is defined by one of next three formulas: 
				\begin{equation}\label{interior_traj_1}
				\bar x_1(t)=x_0+a(t-t_1), \quad t\in [t_0, \alpha],
				\end{equation}
				\begin{equation}\label{interior_traj_2}
				\bar x_1(t)=x_0-a(t-t_1), \quad t\in [t_0, \alpha],
				\end{equation}
				and
				\begin{equation} \label{interior_traj_3}
				\bar x_1(t)=
				\begin{cases}
				 x_0+a(t-t_1), \quad & t\in [t_0, t_{\zeta}]\\
				\bar x_1(t_{\zeta})-a(t-t_\zeta), & t\in (t_{\zeta}, \alpha],
				\end{cases}
				\end{equation}
		where $t_{\zeta} \in (t_0, \alpha)$. Hence, the graph of $\bar x_1(.)$ on $[t_0, \alpha]$ is of one of the following types: 
		C1)~\textit{Going up} as in Fig.~\ref{fig4}; C2) \textit{Going down} as in Fig.~\ref{fig5}; C3) \textit{Going up first and then going down} as in Fig.~\ref{fig6}.
\begin{figure}[!ht] \begin{minipage}[b]{0.5\textwidth}
					      \centering
					      \includegraphics[scale=.15]{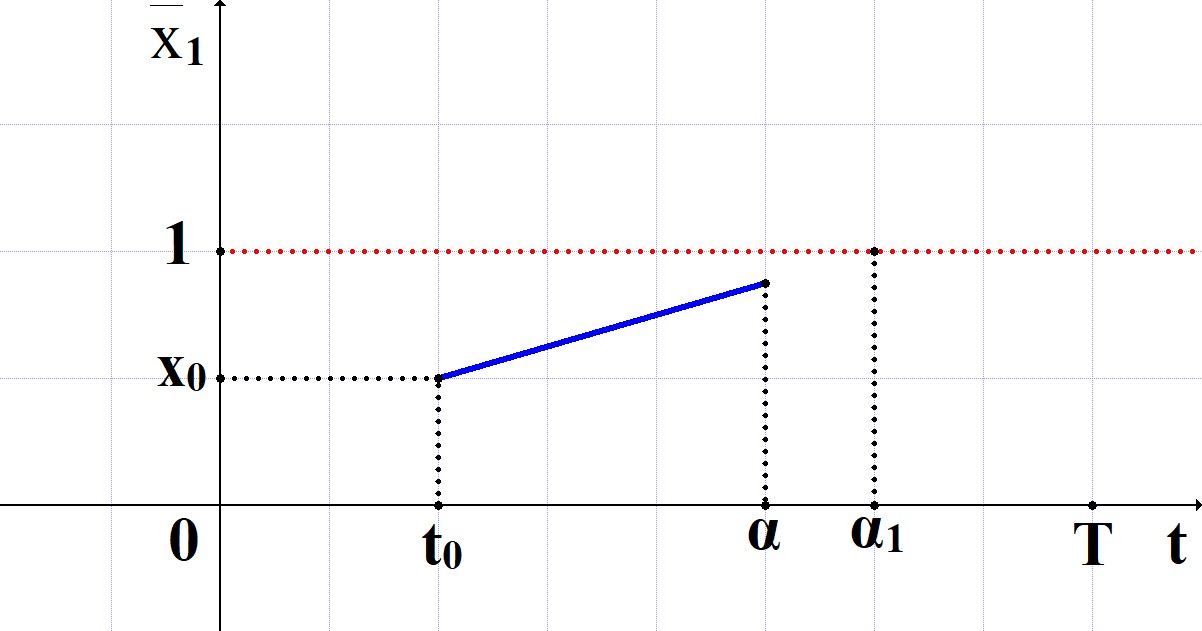}
					      \caption{Category C1}\label{fig4}
					      \end{minipage}
					      \hfill
					      \begin{minipage}[b]{0.5\textwidth}
					      \centering
					      \includegraphics[scale=.15]{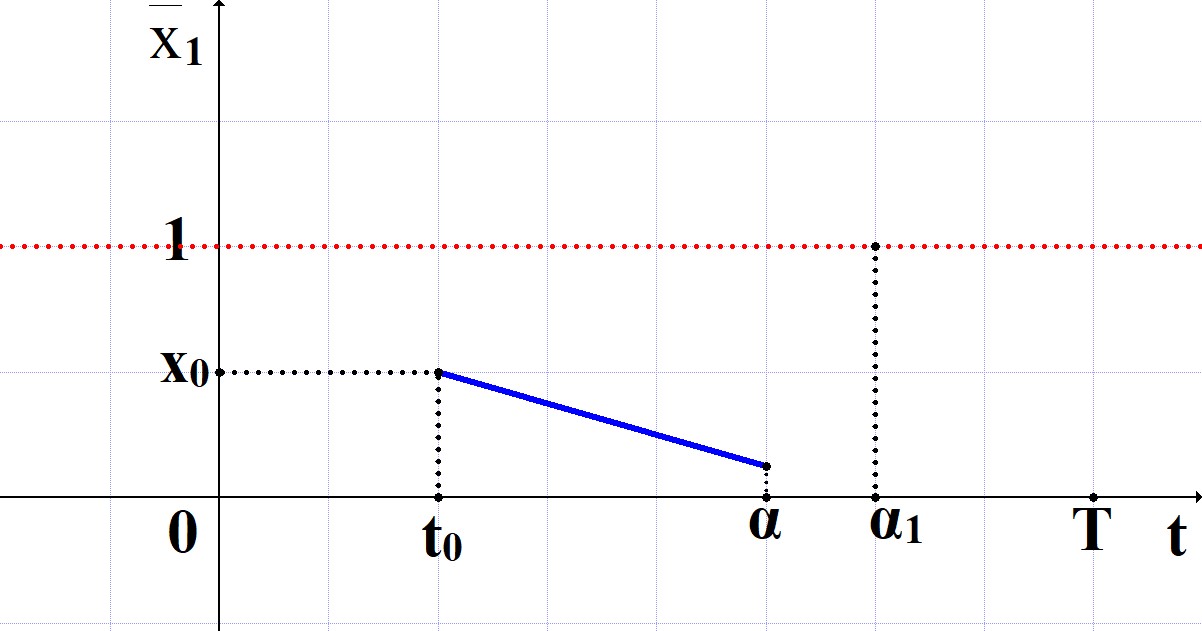}
					      \caption{Category C2}\label{fig5}
					      \end{minipage}
					      \end{figure}	
	 \begin{figure}[!ht]
			\centering
		    \includegraphics[scale=.15]{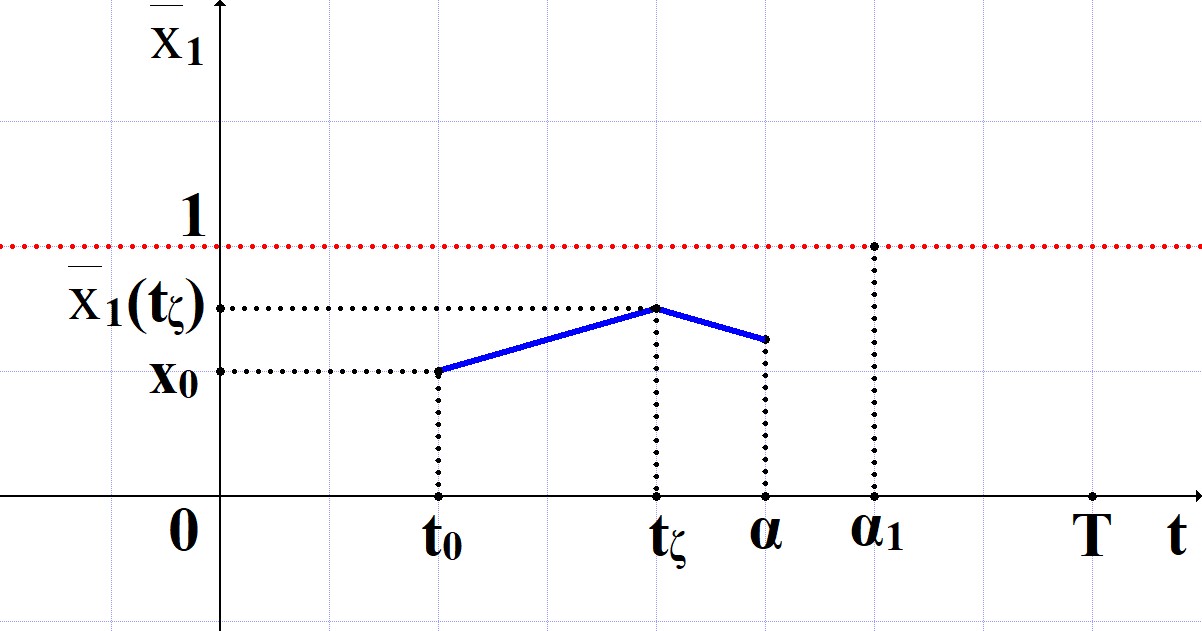}
		     \caption{Category C3}\label{fig6}
			\end{figure}
			Now, let $\alpha=\alpha^{(k)}$ with $\alpha^{(k)}:=\alpha_1-\dfrac{1}{k}$, where $k\in \N$ is as large as  $\alpha\in (t_0, \alpha_1)$. Since for each $k$ the restriction of the graph of $\bar x_1(.)$ on $[t_0,\alpha^{(k)}]$ must be of one of the three types C1--C3, by the Dirichlet principle we can find a subsequence $\{k'\}$ of $\{k\}$ such that the corresponding graphs belong to a fixed category. If the latter is~C2, then by~\eqref{interior_traj_2} and the continuity of $\bar x_1(.)$ one has 
		$$\bar x_1(\alpha_1)=\displaystyle\lim_{k'\to\infty} \bar x_1(\alpha^{(k')})=\displaystyle\lim_{k'\to\infty} \Big[x_0-a(\alpha^{(k')}-t_0)\Big]=x_0-a(\alpha_1-t_0).$$ This is impossible, because $\bar x_1(\alpha_1)=1$. Similarly, the situation where the fixed category is~C3 can be excluded by using \eqref{interior_traj_3}. If the graphs belong to the category~C1, from~\eqref{interior_traj_1} we deduce that
		\begin{equation}\label{interior_traj_1a}
		\bar x_1(t)=x_0+a(t-t_0) \quad (\forall t\in [t_0, \alpha_1]).
		\end{equation} Then, the condition $\bar x_1(\alpha_1)=1$ is satisfied if and only if	$\alpha_1=t_0+a^{-1}(1-x_0).$ 
							
	    \underline{\textit{Subcase 4b}}: $t_0<\alpha_1<\alpha_2<T$. Then, one has $\bar x_1(\alpha_1)=\bar x_1(\alpha_2)=1$ and $\bar x_1(t)<1$ for $t\in [t_0,\alpha_1)\cup (\alpha_2,T]$. We are going to show that this situation cannot occur.  
	    
	    Suppose first that $\bar x_1(t)=1$ for all $t\in (\alpha_1, \alpha_2)$. Since $(\bar x, \bar u)$ is a $W^{1,1}$ local minimizer of $(FP_{2a})$, by Definition \ref{local_minimizer} we can find $\delta>0$ such that the process $(\bar x, \bar u)$ minimizes the quantity $g(x(t_0), x(T))=x_2(T)$ over all feasible processes $(x, u)$ of $(FP_{2a})$ with $\|\bar x-x\|_{W^{1,1}} \leq \delta$. By the result given before Subcase~4a, one has $\bar x_1(t)=1-a(t-\alpha_2)$ for all $t\in [\alpha_2, T]$. Fixing a number $\alpha\in (\alpha_1,\alpha_2)$, we consider the pair of functions $(\widetilde x^\alpha, \widetilde u^\alpha)$ defined by 
	    		\begin{align*}
	    		\widetilde x^\alpha (t):=
	    		\begin{cases}
	    		\bar x(t),  \ \;  t\in [t_0, \alpha)\\
	    		1-a(t-\alpha), \ \; t\in [\alpha, T]
	    		\end{cases}
	    		\quad\; \mbox{and} \quad\;
	    		\widetilde u^\alpha(t):=
	    		\begin{cases}
	    		\bar u(t),  \ \;  t\in [t_0, \alpha)\\
	    		1,  \ \;  t\in [\alpha, T].
	    		\end{cases}
	    		\end{align*}
	    		It is easy to check that $(\widetilde x^\alpha, \widetilde u^\alpha)$ is a feasible process of $(FP_{2a})$. Besides, by direct computing, we have $$\bar x_2(T)-\widetilde x_2^\alpha (T)=\dfrac{1}{\lambda}[a(\alpha-\alpha_2)e^{-\lambda T}+2(e^{-\lambda \alpha}-e^{-\lambda \alpha_2})].$$
	    		Thus, the condition $\alpha<\alpha_2$ yields $\bar x_2(T)>\widetilde x_2^\alpha (T)$. Since  $\displaystyle\lim_{\alpha\to \alpha_2}\|\bar x-\widetilde x^\alpha\|_{W^{1,1}} =0$, one has $\|\bar x-\widetilde x^\alpha\|_{W^{1,1}}\leq\delta$ for all $\alpha\in (\alpha_1,\alpha_2)$ sufficiently close to $\alpha_2$. This contradicts the assumed $W^{1,1}$ local optimality of the process $(\bar x,\bar u)$.
	    		
	    		Now, suppose that there exists $\hat t \in (\alpha_1, \alpha_2)$ such that $\bar x_1(\hat t)<1$. By the continuity of $\bar x_1(.)$, the constants $\hat \alpha_1:=\max\{t\in [\alpha_1, \hat t]\,:\, \bar x_1(t)=1\}$ and $\hat \alpha_2:=\min\{t\in [\hat t, \alpha_2]\,:\, \bar x_1(t)=1\}$ are well defined. Note that $\hat t\in \big(\hat \alpha_1,\hat \alpha_2\big)$ and $\bar x_1(t)<1$ for every $t\in (\hat \alpha_1, \hat \alpha_2)$. If $\varepsilon>0$ is small enough, then $\hat \alpha_1+\varepsilon\in \big(\hat \alpha_1,\hat t\big)$.  Using the result given in Subcase~4a for the restriction of the function $\bar x_1(t)$ on the segment $[\hat \alpha_1+\varepsilon, \hat \alpha_2]$ (thus, $\hat \alpha_1+\varepsilon$ plays the role of $t_0$ and $\hat \alpha_2$ takes the place of $\alpha_1$), one finds that 
	    		\begin{equation*}
	    		\bar x_1(t)=\bar x_1(\hat \alpha_1+\varepsilon)+a(t-\hat \alpha_1-\varepsilon) \quad\; (\forall t\in [\hat \alpha_1+\varepsilon, \hat \alpha_2]).
	    		\end{equation*}
	    		In particular, the function $\bar x_1(t)$ is strictly increasing on $[\hat \alpha_1+\varepsilon, \hat \alpha_2]$. Since $\hat t\in \big( \alpha_1+\varepsilon,\hat \alpha_2\big)$, this implies that $\bar x_1(\hat \alpha_1+\varepsilon)<\bar x_1(\hat t)$. Then, by the continuity of $\bar x_1(t)$ we obtain
	    		\begin{equation*}
	    		\bar x_1(\hat \alpha_1)=\displaystyle\lim_{\varepsilon\to 0} \bar x_1(\hat \alpha_1+\varepsilon)\leq\bar x_1(\hat t)<1.
	    		\end{equation*}
	    		As $\bar x_1(\hat \alpha_1)=1$, we have arrived at a contradiction.
	    		
	    	Since Subcase~4b cannot happen, we conclude that the formula for $\bar x_1(t)$ in this case is given by 
	    		 \begin{equation*}
	    		\bar x_1(t)=
	    		\begin{cases}
	    		x_0+a(t-t_0), \quad & t \in [t_0, \alpha_1]\\
	    		1-a(t-\alpha_1), & t \in (\alpha_1, T],
	    		\end{cases}
	    		\end{equation*}
	    	    with $\alpha_1:=t_0+a^{-1}(1-x_0)$.  One must have $\alpha_1\leq \bar t$, where $\bar t$ is defined by \eqref{special_time}. Indeed, suppose on the contrary that $\alpha_1>\bar t$. For an arbitrarily given $\alpha \in (\bar t, \alpha_1)$, we consider the problem $(FP_{1b})$ (resp.,  the problem $(FP_{2b})$) which is obtained from the problem $(FP_{1a})$ in Section~\ref{Example 1} (resp., from the above problem $(FP_{2b})$) by letting  $\alpha$ play the role of the initial time $t_0$. Since $\alpha > \bar t$, it follows from Theorem~\ref{Thm1} that $(FP_{1b})$ has a unique global solution $(\bar x^\alpha,\bar u^\alpha)$, where $\bar u(t)=1$ for almost everywhere $t\in [\alpha, T]$, $\bar x_1^\alpha(t)=\bar x_1(\alpha)-a(t-\alpha)$ for all $t \in [\alpha, T]$, and $\bar x_2^\alpha(t)= \int_{\alpha}^{t} \big[-e^{-\lambda\tau}(x_1(\tau)+u(\tau))\big] d\tau$ for all $t \in [\alpha, T]$. Clearly, the restriction of $(\bar x, \bar u)$ on $[\alpha, T]$ is a feasible process for $(FP_{1b})$. Thus, we have 
	    	    \begin{equation}\label{intergral_inequality}
	    	    \bar x_2^\alpha(T)<\bar x_2(T).
	    	    \end{equation}
	    		Besides, by Proposition~\ref{lemma_basic_property_2}, the restriction of $(\bar x, \bar u)$ on $[\alpha, T]$ is a $W^{1,1}$ local solution for $(FP_{2b})$. So, there exits $\delta>0$ such that the restriction of $(\bar x, \bar u)$ on $[\alpha, T]$ minimizes the quantity $x_2(T)$ over all feasible processes $(x, u)$ of $(FP_{2b})$ with $\|x-\bar x\|_{W^{1,1}([\alpha, T]; \R^n)} \leq \delta$. Clearly,  $(\bar x^\alpha,\bar u^\alpha)$ is a feasible process of $(FP_{2b})$. Therefore, since  $\|\bar x^\alpha -\bar x\|_{W^{1,1}([\alpha, T]; \R^n)} \leq \delta$ for all $\alpha$ sufficiently close to $\alpha_1$, we have $\bar x_2^\alpha(T)\geq\bar x_2(T)$ for those $\alpha$. This contradicts~\eqref{intergral_inequality}.
	    			    		
	    		\medskip
	    		Going back to the original problem $(FP_2)$, we can summarize the results obtained in this section as follows.
	    		
	    		\begin{Theorem}\label{Thm2}  Given any $a,\lambda$ with $a>\lambda>0$, define $\rho=\dfrac{1}{\lambda}\ln \dfrac{a}{a-\lambda}>0$, $\bar t=T-\rho$ $\bar x_0=1-a(\bar t-t_0)$, and $\alpha_1=t_0+a^{-1}(1-x_0)$. Then, problem $(FP_2)$  has a unique local solution $(\bar x,\bar u)$, which is a global solution, where $\bar u(t)=-a^{-1}\dot{\bar x}(t)$ for almost everywhere $t\in [t_0, T]$ and $\bar x(t)$ can be described as follows: 
	    	    \begin{description}
	    	    \item{\rm (a)} If $t_0 \geq \bar t$ (i.e, $T-t_0 \leq \rho$), then 
	    	    \begin{equation*}
	    	     \bar x(t)=x_0-a(t-t_0), \quad t \in [t_0, T].
	    	     \end{equation*}
	    		 \item{\rm (b)} If $t_0 < \bar t$ and $x_0<\bar x_0$ (i.e, $\rho< T-t_0<\rho+a^{-1}(1-x_0)$), then 
	    		 \begin{equation*}
	    		 \bar x(t)=
	    		 \begin{cases}
	    		 x_0+a(t-t_0), \quad & t \in [t_0, \bar t]\\
	    		 x_0-a(t+t_0-2\bar t), & t \in (\bar t, T].
	    		 \end{cases}
	    		 \end{equation*}
	    		 \item{\rm (c)} If $t_0 < \bar t$ and $x_0\geq \bar x_0$ (i.e, $T-t_0\geq\rho+a^{-1}(1-x_0)$), then 
	    		 \begin{equation*}
	    		 \bar x(t)=
	    		 \begin{cases}
	    		  x_0+a(t-t_0), \quad & t \in [t_0, \alpha_1]\\
	    		 1-a(t-\alpha_1), & t \in (\alpha_1, T].
	    		 \end{cases}
	    		 \end{equation*}
	    		\end{description}
	    		\end{Theorem}
	    		\begin{proof}  To obtain the assertions (a)--(c), it suffices to combine the results formulated in Case~1, Case~3, and Case~4, having in mind that $\bar x_1(t)$ in $(FP_{2a})$ plays the role of $\bar x(t)$ in $(FP_2)$.
	    			  	\end{proof}	    		
	  	  	
	    \section{Conclusions}\label{Conclusions}
	    
	    We have analyzed a maximum principle for finite horizon state constrained problems via two parametric examples of optimal control problems of the Langrange type, which have five parameters. These problems resemble the optimal growth problem in mathematical economics. The first example is related to control problems without state constraints. The second one belongs to the class of irregular control problems with unilateral state constraints. We have proved that the control problem in each example has a unique local solution, which is a global solution. Moreover, we are able to present an explicit description of the optimal process with respect to the five parameters.
	    
	    The obtained results allows us to have a deep understanding of the maximum principle in question. 
	    
	    It seems to us that, following the approach adopted in this paper, one can study economic optimal growth models by advanced tools from functional analysis and optimal control theory.

\end{document}